\documentclass[11 pt,reqno]{amsart}

\usepackage[T1]{fontenc}
 \usepackage[utf8]{inputenc}
\usepackage{amsfonts,amssymb,amsmath,amsthm}
\usepackage[headinclude,DIV13]{typearea}
\usepackage{hyperref}
\usepackage{geometry}
\usepackage{graphicx, psfrag}
\usepackage{supertabular}
 \usepackage{color}
 \usepackage{setspace}

\newtheorem{rem}{Remark}
\newtheorem{lem}{Lemma}[section]
\newtheorem{prop}{Proposition}[section]

\newtheorem{theo}{Theorem}

\theoremstyle{definition} 
\theoremstyle{definition} 

\newcommand{\norm}[2]{\lvert\lvert#1\rvert\rvert_{#2}}
\renewcommand{\P}{\mathbb{P}}
\newcommand{\R}{\mathbb{R}}

\newcommand{\E}{\mathbb{E}}

\newcommand{\N}{\mathbb{N}}

\newcommand{\A}{\mathcal{A}}

\newcommand{\zbf}{{\bf z}}
\newcommand{\Zbf}{{\bf Z}}
\newcommand{\nbf}{{\bf n}}
\newcommand{\Nbf}{{\bf N}}

\newcommand{\wt}[1]{\widetilde{#1}} 

\newcommand{\eps}{\varepsilon}

\newcommand{\ben}{\vspace{0mm}\begin{equation}}
\newcommand{\bea}{\vspace{0mm}\begin{equation*}\begin{aligned}}
\newcommand{\eea}{\vspace{0mm}\end{aligned}\end{equation*}}
\newcommand{\een}{\vspace{0mm}\end{equation}}

\newcommand{\proportion}{\rho}

\usepackage{ulem}

\numberwithin{equation}{section}

\definecolor{orange}{RGB}{255,107,20}

\begin{document}
\title[Emergence of homogamy in a stochastic model]
{Emergence of homogamy in a two-loci stochastic population model}


\author{Camille Coron}
\address{Laboratoire de Math\'ematiques d'Orsay, Univ. Paris-Sud, CNRS, Universit\'e Paris-Saclay, 91405 Orsay, France}
\email{camille.coron@math.u-psud.fr}

\author{Manon Costa}
\address{Institut de Math\'ematiques de Toulouse.  CNRS UMR 5219, Universit\'e Toulouse 3 - Paul Sabatier, 118 route de Narbonne, F-31062 Toulouse cedex 09}
\email{manon.costa@math.univ-toulouse.fr}

\author{Fabien Laroche}
\address{Irstea, UR EFNO, Domaine des Barres, 45290 Nogent-sur-Vernisson, France}
\email{fabien.laroche@irstea.fr}

\author{H\'el\`ene Leman}
\address{Universit\'e de Lyon, Inria, CNRS, ENS de Lyon, UMPA UMR 5669, 46 all\'ee d’Italie, 69364 Lyon, France}
\email{helene.leman@inria.fr}

\author{Charline Smadi}
\address{IRSTEA UR LISC, Laboratoire d'ing\'enierie des Syst\`emes Complexes, 9 avenue Blaise-Pascal CS
20085, 63178 Aubi\`ere, France and Complex Systems Institute of Paris île-de-France (ISC-PIF, UPS3611), 113 rue Nationale, Paris, France}
\email{charline.smadi@irstea.fr}

\maketitle

\begin{abstract}
 This article deals with the emergence of a specific mating preference pattern called homogamy in a population. Individuals are characterized by their genotype at two haploid loci, and the population dynamics is modelled by a non-linear birth-and-death process. The first locus codes for a phenotype, while the second locus codes for homogamy defined with respect to the first locus: two individuals are more (resp. less) likely to reproduce with each other if they carry the same (resp. a different) trait at the first locus. Initial resident individuals do not feature homogamy, and we are interested in the probability and time of invasion of a mutant presenting this characteristic under a large population assumption. To this aim, we study the trajectory of the birth-and-death process during three phases: growth of the mutant, coexistence of the two types, and extinction of the resident. We couple the birth-and-death process with simpler processes, like multidimensional branching processes or dynamical systems, and study the latter ones in order to control the trajectory and duration of each phase.
\end{abstract}

\vspace{0.2in}


\bigskip


\vspace{0.5cm}

\section{Introduction and motivation}\label{sec:intro}
Assortative mating is a mating pattern in which individuals with similar phenotypes reproduce more frequently than expected 
under uniform random mating. Such a reproductive behaviour is widespread in natural populations and has an important role in the shape of 
their evolution (see for instance \cite{mclain1987male,Herrero2003,Savolainenetal2006} or the review \cite{jiang2013assortative} 
on assortative mating in animals). In particular assortative mating is expected to be a driving force for speciation, which is the process 
by which several species arise from a single one \cite{gregorius1992two}.
Here we ask the question of assortative mating emergence in 
a population: if one mutant starts mating preferentially with individuals of the same type, while the other individuals 
still choose their mate uniformly at random, can 
this mutant invade the population?
A key feature to answer this question is how the assortative mating 
mutation affects the total reproduction rate of individuals. 
The existence of a preference for a given phenotype is often associated with a decay of reproductive success when mating with other phenotypes.
As a consequence, if the proportion of preferred individuals is low in the population, the assortative mating mutation may be 
detrimental because it decreases the total reproduction success. Consequently, we expect that an assortative mating mutant will be able to invade only 
if its choosiness is compensated by an increased number of potential mates or if the advantage given by the preference is high enough. 

In this work we aim at quantifying the conditions on the trade-off between advantage and cost for assortative mating and on the phenotype composition of the existing population needed for the mutation to invade the population.
In order to reach this goal, we build a stochastic individual-based population model with varying size, which explores how the relationship between increase in the number of mates and mating bias towards individuals of the same type 
affects the long time number of individuals having mating preferences in the population.

The class of stochastic individual-based models with competition and varying population size we are extending have been introduced in the 90's 
in~\cite{bolker1997using,dieckmann1996dynamical} and made rigorous in a probabilistic setting in the seminal paper of Fournier 
and M\'el\'eard \cite{fournier2004microscopic}. 
Initially restricted to asexual populations, such models have evolved to incorporate the case of sexual reproduction, in both haploid 
\cite{smadi2015eco,leman2018stochastic} and diploid \cite{collet2011rigorous,coron2015slow,neukirch2017survival,smadi2018looking} populations. 
Taking into account varying population sizes and stochasticity is necessary if we aim at better understanding phenomena involving small populations like invasion 
of a mutant population \cite{champagnat2006microscopic} or population extinction time.
Assuming that individuals initially have no preference and choose their mate uniformly at random, we suppose that a mutation arises in the population: individuals carrying the mutation 
(denoted $P$) have a higher (resp. smaller) reproductive success when mating with individuals of the same (resp. different) phenotype 
than individuals without the mutation. 
We study under which conditions on the parameters (birth and death rates, competition, mutational effects, initial population state, ...) the mutation $P$ has a positive probability to 
invade the population, and how to identify this probability. We also characterize the time needed for the mutation to get fixed in the population when it happens. 
Finally, we provide the invasion dynamics as well as the final population state, when the mutation gets fixed.

In order to obtain our results, we study the population process at two different scales. 
When one sub-population is of small size the stochasticity of 
its size has a major effect on the population long time behaviour, and we study its dynamics on $\mathbb{N}:= \{0,1,2,...\}$. 
This is for example the case of the mutant 
population when it arises. When on the contrary all sub-populations sizes are large, we approximate the stochastic process by a mean field limit which is a 
dynamical system.

Note that the study of the population process is more involved than in the previous references on similar questions 
(see for instance \cite{champagnat2006microscopic,champagnat2011polymorphic,billiard2017interplay}) because the initial state of the population is not 
an hyperbolic equilibrium, since alleles $A$ and $a$ are initially neutral. As a consequence, the fluctuations around the initial state may be 
substantial and are strongly influenced by the presence of mutants, even in a small number. We thus cannot use the classical large deviation theory \cite{dupuis2011weak}, 
and we need to study the dynamics of the types altogether.
Moreover, again unlike in \cite{champagnat2006microscopic,champagnat2011polymorphic,billiard2017interplay} but similarly as in \cite{Coron2018}, the dynamical 
system arising as the limit of the rescaled population after the invasion phase admits many (stable and unstable) fixed points and we need 
to identify precisely the four dimensional zone reached by the rescaled population process after the invasion phase in order to 
determine the convergence point of the dynamical system.

\section{Model and main results}

We consider a population of individuals that reproduce sexually and compete with each other for a common resource. Individuals are haploid and are characterized by their genotype at 
two loci located on different chromosomes. Locus $1$ presents two alleles, denoted by $A$ and $a$, 
and codes for phenotypes. Locus $2$ presents two alleles denoted by 
$P$ and $p$, and codes for 
assortative mating, which is defined relatively to the first locus (similar models were introduced in Biology, see for example \cite{gregorius1992two}). More precisely, we assume that all individuals try to reproduce at the same rate. To this aim, they 
choose a mate, uniformly at random among the other individuals of the population. Next, individuals carrying allele $p$ reproduce indifferently with their 
chosen partner, while individuals carrying allele $P$ reproduce with a higher probability with individuals carrying the same allele at locus $1$. Note that 
reproduction is not completely symmetric: only the genotype of the individual initiating the reproduction determines the presence or not of assortative mating.

The genotype of each individual belongs to the set $\mathcal{G} :=\{AP,Ap,aP,ap\}$ and the state of the population is characterized at each time $t$ by a vector 
in $\N^4$ giving the respective numbers of individuals carrying each of these four genotypes. The dynamics of this population 
is modeled by a multi-type birth-and-death process 
$$(N(t), t \geq 0):=(N_{AP}(t),N_{Ap}(t),N_{aP}(t),N_{ap}(t), t\geq 0)$$
with values in $\N^4$, integrating competition, 
Mendelian reproduction and assortative mating. More precisely, when the population is in state 
${\bf n}=(n_{AP},n_{Ap},n_{aP},n_{ap})\in\N^4$ with size 
$n=n_{AP}+n_{Ap}+n_{aP}+n_{ap}$, 
then the rate at which the population looses an individual with genotype $i\in \mathcal{G}$, is equal to
\ben 
\label{eq:mort}
d_i(\nbf)=n_i\left(d+\frac{c}{K}n\right).\een 
The parameters $d\in\mathbb{R}^+$, $c>0$ and $K>0$  respectively model the natural and the competition death rates of individuals 
and a scaling parameter of the total population size.
This parameter quantifies the environment's carrying capacity, which is a measure of the
maximal population size that the environment can sustain for a long time.
In the sequel we will be interested in the behaviour of the system for large but finite $K$.

When the population is in state $\nbf$, 
the rate $b_i(n)$ at which an individual with genotype $i\in \mathcal{G}$ is born, is defined by

\begin{equation}\label{eq:naissAP}
\begin{aligned}
 b_{AP}(\nbf)&= b\left[ n_{AP}+ \frac{1}{n}\left( \beta_1 n_{AP}\left( n_{AP}+\frac{n_{Ap}}{2} \right)-\beta_2\left( n_{AP} \left( n_{aP}+\frac{n_{ap}}{4}\right)+ n_{Ap}
 \frac{n_{aP}}{4} \right) \right)
 + \frac{\Delta_{aP}}{2n} \right] \\
 b_{Ap}(\nbf)&= b \left[ n_{Ap}+ \frac{1}{n}\left( \beta_1 n_{Ap}\frac{n_{AP}}{2}-\beta_2\left( n_{Ap}  \frac{n_{aP}}{4}+ n_{AP}
 \frac{n_{ap}}{4} \right) \right)
 - \frac{\Delta_{aP}}{2n} \right] \\
 b_{aP}(\nbf)&= b \left[ n_{aP}+ \frac{1}{n}\left( \beta_1 n_{aP}\left( n_{aP}+\frac{n_{ap}}{2} \right)-\beta_2\left( n_{aP} \left( n_{AP}+\frac{n_{Ap}}{4}\right)+ n_{ap}
 \frac{n_{AP}}{4} \right) \right)
 - \frac{\Delta_{aP}}{2n} \right]  \\
 b_{ap}(\nbf)&= b \left[ n_{ap}+ \frac{1}{n}\left( \beta_1 n_{ap}\frac{n_{aP}}{2}-\beta_2\left( n_{ap}  \frac{n_{AP}}{4}+ n_{aP}
 \frac{n_{Ap}}{4} \right) \right)
 + \frac{\Delta_{aP}}{2n} \right],
\end{aligned}
\end{equation}
where 
$$\Delta_{aP}:= n_{aP}n_{Ap}-n_{AP}n_{ap}.$$
The parameter $b(1+\beta_1)$ with $b>0$ and $\beta_1\geq 0$ is the rate at which any individual (called first parent) reproduces, and each reproduction leads 
to the birth of a new individual with probability $1/(1+\beta_1)$ when the first parent carries allele $p$, with probability $1$ if the first 
parent carries allele $P$ and both parents carry the same allele at locus $1$, and with probability $(1-\beta_2)/(1+\beta_1)$ if the first 
parent carries allele $P$ and the two parents carry different alleles at locus $1$. The parameters $\beta_1$ and $\beta_2$ respectively quantify benefits and penalties for homogamous individuals. Table \ref{tab:repro} in Appendix \ref{appTable} 
summarizes the different rates at which a pair of parents with given genotypes gives birth to an offspring with a given genotype.
This explains how the birth rates \eqref{eq:naissAP} are obtained.\\

Throughout the paper, we will make the following assumptions on the parameters:
\begin{enumerate}
 \item $b>d$
 \item $\beta_1\geq 0$
 \item $0 \leq \beta_2 \leq 1$
\end{enumerate}
The first assumption ensures that a population of individuals mating uniformly at random is not doomed to a rapid extinction because of a natural death 
rate larger than the birth rate under uniform random mating. The second (resp. third) assumption means that choosy individuals have a higher 
(resp. smaller) probability to give birth when mating with an individual with the same (resp. different) trait ($A$ or $a$).\\

We assume that at time $0$, all individuals mate uniformly at random (no sexual preference, all individuals carry allele $p$), and that the population size is close to its long time equilibrium, $(b-d)K/c$ (see on page \pageref{cond_ini} for details).
A mutant (or a migrant) appears in the population, with genotype $\alpha P$, where $\alpha \in \mathfrak{A}:=\{A,a\}$. The goal of our main theorem (Theorem \ref{theo_main}) is to study a step in Darwinian evolution, that consists in the progressive invasion of the new allele $P$ and loss of initial allele $p$ in the population. The proof of this theorem relies on the study of three phases in the population dynamics trajectories (mutant survival or extinction, mean-field phase, and resident allele extinction) that are respectively defined and studied in Subsections \ref{subsectionInvasion}, \ref{subsectionMeanField}, \ref{section_ext}. The statement of Theorem \ref{theo_main} requires the introduction of several quantities that we define now.

\medskip
Our first goal is to determine conditions under which the mutant population has a positive probability to survive and invade the resident population. In order to answer this question, we will compare the mutant population with a branching process during the first times 
of the invasion.
This comparison follows from the following observation that will be proved in Proposition \ref{prop_equiv_inv}: 
as long as the mutant population size is negligible with respect to the carrying capacity $K$, the dynamics of the resident population will not be affected by the presence of the mutants and will stay close to its initial state. In other words, the size and proportions of the resident population will remain almost constant and the dynamics of the mutant population will be close to the dynamics of the process
 $\bar{\Nbf}=(\bar{N}_A,\bar{N}_a)$, which is a bi-type branching process with the following transition rates:  
\begin{equation}\begin{aligned}
\label{defbarN}
 (\bar{N}_A,\bar{N}_a) \to (\bar{N}_A+1,\bar{N}_a)  \quad &\text{at rate} \quad \bar{\beta}_{AA}\bar{N}_A+\bar{\beta}_{aA}\bar{N}_a \\ 
 (\bar{N}_A,\bar{N}_a) \to (\bar{N}_A,\bar{N}_a+1) \quad &\text{at rate} \quad \bar{\beta}_{Aa}\bar{N}_A+\bar{\beta}_{aa}\bar{N}_a \\ 
 (\bar{N}_A,\bar{N}_a) \to (\bar{N}_A-1,\bar{N}_a) \quad &\text{at rate} \quad b\bar{N}_A \\ 
 (\bar{N}_A,\bar{N}_a) \to (\bar{N}_A,\bar{N}_a-1) \quad &\text{at rate} \quad b\bar{N}_a ,
\end{aligned}\end{equation}
where for $\alpha \in \mathfrak{A}$, $\bar{\alpha}  \in  \mathfrak{A} \setminus\{ \alpha\}$,
\begin{equation} \label{transi_rates} \bar{\beta}_{\alpha\alpha}:= \frac{b}{2} \left( 1+ (\beta_1+1)\proportion_\alpha- \frac{\beta_2}{2}\proportion_{\bar{\alpha}} \right), \quad \bar{\beta}_{\alpha \bar{\alpha}}:= \frac{b}{2} \left(1- \frac{\beta_2}{2} \right)\proportion_{\bar{\alpha}}, \end{equation}
and
\begin{equation}\label{defpropinitiale}
\proportion_A:= \frac{N_{Ap}(0)}{N_p(0)} \quad \text{and} \quad \proportion_a:=1-\proportion_A=\frac{N_{ap}(0)}{N_p(0)} 
\end{equation}
are the initial proportions in the resident population. The rates of this branching process have been obtained by considering the dynamics of $(N_{AP}, N_{aP})$ 
described by \eqref{eq:mort} and \eqref{eq:naissAP} 
when $(N_{Ap}, N_{ap})=(K\proportion_A \frac{b-d}{c}, K(1-\proportion_A)\frac{b-d}{c})$, $N=K\frac{b-d}{c}$ and the second order terms in $N_{AP}$ and $N_{aP}$ are neglected.
We denote the extinction probabilities of the process $\bar{\Nbf}$ by 
\begin{equation}\label{def_probaq}
q_\alpha:= \P( \exists t < \infty,  \bar{\Nbf}(t)=0 | \bar{\Nbf}(0)= \mathbf{e}_\alpha),
\end{equation}
$\alpha \in \mathfrak{A}$, $\mathbf{e}_A=(1,0)$ and $\mathbf{e}_a=(0,1)$, meaning that the process starts with only one individual of type $A$ or $a$.
Classical results of branching process theory (see \cite{athreya1972branching}) ensure that these extinction probabilities 
correspond to the smallest solution to the system of equations
\begin{align} \label{syst_prob_ext}
 u_{A}(s_A,s_a)&:= b(1-s_A)+ \bar{\beta}_{AA}(s_A^2-s_A)+ \bar{\beta}_{Aa}(s_A s_a-s_A)=0 \\
 u_{a}(s_A,s_a)&:= b(1-s_a)+ \bar{\beta}_{aa}(s_a^2-s_a)+ \bar{\beta}_{aA}(s_A s_a-s_a)=0. \nonumber
\end{align}
Moreover, the process $\bar{\bf N}$ is supercritical (and in this case $ q_A$ and $q_a$ are not equal to one) if and only if the following
matrix, which corresponds to the infinitesimal generator of the branching process, has a positive eigenvalue 
\begin{equation}\label{matrix} 
J:=\left( \begin{array}{cc}
           \bar{\beta}_{AA}- b & \bar{\beta}_{Aa}\\
	    \bar{\beta}_{aA} & \bar{\beta}_{aa}-b
           \end{array}
 \right) 
 \end{equation}
that is to say if and only if
\begin{equation}
\label{eq:noStabNoMig}
\beta_1 > \beta_2 \quad \text{or} \quad \proportion_A(1-\proportion_A) < \frac{\beta_1 (\beta_2+2)}{2(\beta_1+\beta_2)(\beta_1+2)}
\end{equation}
(see the proof of Proposition \ref{prop:invasibilite}). We denote by $\lambda$ the maximal eigenvalue of \eqref{matrix}, which is thus positive when \eqref{eq:noStabNoMig} holds and which will be of interest to quantify the time before invasion.
Notice that $J$ can be written as $b$ times a matrix only depending on $(\proportion_A,\beta_1,\beta_2)$. As a consequence, 
$\lambda$ can be written $\lambda= b\tilde{\lambda}(\proportion_A,\beta_1,\beta_2)$. We will use this notation in Theorem \ref{theo_main} to make appear 
the dependence on the parameters, and use $\lambda$ elsewhere for the sake of readability.\\

If the mutant population invades and its size reaches order $K$ with $K$ large, the population dynamics enters a second phase during which it is well approximated (see Proposition \ref{prop:largepop} for a rigorous statement) by a mean field process. More precisely, if we define the rescaled process
 $$(\mathbf{Z}^K(t), t\geq 0):=\left(\frac{N_{AP}(t)}{K},\frac{N_{Ap}(t)}{K},\frac{N_{aP}(t)}{K},\frac{N_{ap}(t)}{K}, t\geq 0\right),$$
then it will be close to the solution of the dynamical system
\begin{equation}
\label{eq:syst}
\dot{z}_i= b_i(\mathbf{z})- (d+c z)z_i, \quad i \in \mathcal{G},
\end{equation} 
where $z=z_{AP}+z_{Ap}+z_{aP}+z_{ap}$ is the total size of the population and the functions $b_i$ have been defined in Equation \eqref{eq:naissAP}. This dynamical system has a unique solution, as the vector field is locally Lipschitz and that the solutions do not explode in finite time \cite{chicone2006ode}.
If we denote by 
$$(\mathbf{z}^{(\mathbf{z}^0)}(t),t\ge0)=(z_{AP}(t),z_{Ap}(t),z_{aP}(t),z_{ap}(t), t\geq 0)$$
this unique solution starting 
from $\mathbf{z}(0)=\mathbf{z}^0\in\R_+^4$, we have the following result, which derives from 
Theorem 2.1 p 456 in \cite{Ethier-Kurtz}.
\begin{lem}
\label{prop:largepop} Let $T\in\R^*_+$. Assume that the sequence $(\mathbf{Z}^K(0), K\ge1)$ converges in probability to some deterministic 
vector $\mathbf{z}^0=(z_{AP}(0),z_{Ap}(0),z_{aP}(0),z_{ap}(0))\in\R_+^4$ when $K$ goes to infinity. Then 
$$\lim_{K\to\infty}\sup_{s\le T} ||\mathbf{Z}^K(s)-\mathbf{z}^{(\mathbf{z}^0)}(s)||_{\infty}=0 \quad\text{in probability},$$
where $||\cdot||_{\infty}$ denotes the $L^{\infty}$-Norm in $\R^4$.
\end{lem}

Notice that when there are only individuals of type $p$ in the population (no sexual preferences), the dynamical system \eqref{eq:syst} is
$$\left\{\begin{aligned}
&\dot{z}_{Ap}=z_{Ap}(b-d-c(z_{Ap}+z_{ap}))\\
&\dot{z}_{ap}=z_{ap}(b-d-c(z_{Ap}+z_{ap}))
\end{aligned}\right.
$$
This system admits an infinity of equilibria:
\begin{itemize}
\item $(z_{Ap},z_{ap})=(0,0)$, which is unstable
\item $(z_{Ap},z_{ap})=(\proportion(b-d)/c,(1-\proportion)(b-d)/c)$ for all $\proportion\in [0,1]$, which are non hyperbolic.
\end{itemize}
However, if we consider the equation giving the dynamics of the total population size $z=z_{Ap}+z_{ap}$, we get 
$$ \dot{z}= z(b-d-cz). $$
Its solution, with a positive initial condition, converges to its
unique stable equilibrium, $(b-d)/c$. That is why we will assume that the initial population size, before the arrival of the mutant, is $(b-d)K/c$. \label{cond_ini}\\

A fine study of the dynamics of the solutions to \eqref{eq:syst} with our particular initial conditions, that is to say few individuals mating assortatively at 
the beginning and a majority of $A$ (or $a$) in both resident and mutant populations (see Section \ref{section_convsystdyn}), will allow us to show that the dynamical system converges to an equilibrium where some of the variables $z_i, i \in \mathcal{G}$ 
are equal to $0$. When the population size of these $i$ becomes too small (of order smaller than $K$ before rescaling), 
the mean fields approximation stops being a good approximation, and we will again compare the dynamics of the small population sizes with these of branching processes (now subcritical). The birth and death rates of these branching processes 
will provide the time to extinction of these small populations (see Section \ref{section_ext}).\\

Combining all these steps, we are able to describe the invasion/extinction dynamics of the mutant population, which is the subject of 
the main result of this paper, Theorem \ref{theo_main}.
Before stating it, we need to introduce some last notations: a set of interest for the rescaled process $\mathbf{Z}^K$, for any $\mu>0$
\begin{equation}\label{defSeps}
 S_{\mu} := \left[ \frac{b(1+\beta_1)-d}{c}-\mu , \frac{b(1+\beta_1)-d}{c}+\mu \right] \times \{0\} \times \{0\}\times \{0\}, 
 \end{equation}
a stopping time describing the time at which $\mathbf{Z}^K$ reaches this set,
\begin{equation} \label{defTKSeps} T_{S_{\mu}}:= \inf \{ t \geq 0, \mathbf{Z}^K(t) \in S_{\mu} \}. \end{equation}
as well as a stopping time which gives the first time when the rescaled
$P$-mutant population size reaches any threshold (from below or above): for any $\eps \geq 0$, 
\begin{equation} \label{defTNT0}
 \quad  T^P_{\eps}:= \inf \left\{ t > 0, N_P(t) = \lfloor \eps K \rfloor \right\},
\end{equation}
where $\lfloor x \rfloor$ is the integer part of $x$.

\begin{theo} \label{theo_main}
Assume that $\lambda \neq 0$,
$$ \left( Z^K_{Ap}(0),Z^K_{ap}(0) \right) \underset{K \to \infty}{\to} 
\left( \proportion_A\frac{b-d}{c},(1-\proportion_A)\frac{b-d}{c} \right) $$
in probability with $\proportion_A>1/2$ and that for any $\alpha \in \mathfrak{A}$
$$ \left( N_{\alpha P}(0),  N_{\bar{\alpha} P}(0) \right)= (1,0). $$
Then there exists a Bernoulli random variable with parameter $1-q_\alpha$, $B$, such that for any $0 <\mu < (b(1+\beta_1)-d)/c $:
\begin{equation} \label{res_main} \lim_{K \to \infty} \left( \frac{T_{S_\mu} \wedge T_0^P}{\ln K}, \mathbf{1}_{\{ T_{S_\mu} < T_0^P \}} \right) = 
B\times\left(  \frac{1}{b\tilde{\lambda}(\proportion_A,\beta_1,\beta_2)}+ \frac{2}{b\beta_1}, 1 \right), \end{equation}
where the convergence holds in probability.\\
Moreover, 
\begin{equation}\label{eq_cv_extin}\mathbf{1}_{\{ T_0^P < T_{S_\mu} \}} \left|\left| \frac{\Nbf(T_0^P)}{K}- (0,\rho_A,0,1-\rho_A)\frac{b-d}{c} \right|\right|_{1}\underset{K \to \infty}{\longrightarrow} 0\hspace{1,5cm}\text{in probability}, 
\end{equation} 
where $\|\cdot\|_1$ stands for the $L^1-$norm.
\end{theo}

Notice that if condition \eqref{eq:noStabNoMig} does not hold, $q_\alpha=1$, and the convergence in \eqref{res_main} is an almost sure convergence 
to $(0,0)$ meaning that the mutant population dies out in a time smaller than $\ln K$.
In this case, the allelic proportions in the resident population do not vary.
Condition \eqref{eq:noStabNoMig} gives two possible sufficient conditions for the mutant population to invade with positive probability. The first one imposes that the trade-off between the advantage for homogamous 
reproduction ($\beta_1$) and the loss for heterogamous reproduction ($\beta_2$) has to be favourable enough. The second condition requires a low level of initial
allelic diversity at locus 1 (alleles $A$ and $a$). In particular, even if the advantage for homogamy is very low, very asymmetrical initial conditions ($\rho_A$ close to 0 or 1) will ensure the invasion of the mutation with positive probability.
As expected, these conditions are the same as the conditions for the approximating branching process $\bar{\textbf{N}}$ defined on page \pageref{defbarN} to be supercritical. In fact, as we will see later in the proof, the random variable $B$ will be the indicator of survival of a version of $\bar{\mathbf{N}}$ coupled with the mutant process.

Let us emphasize that our result ensures that when the mutant population invades (whatever allele $a$ or $A$ the first mutant carries), then the final population is monomorphic, 
and all individuals carry the allele $a$ or $A$ which was in the majority in the resident $p-$population. Only the mutant invasion probability depends on the allele carried by the first $P$ individual.\\

We were not able to obtain an explicit formula in general for the extinction probability $q_\alpha$ 
of the assortative mating mutation, solutions of \eqref{syst_prob_ext}. 
However, in the particular case when there are only $A$ or $a$-individuals in the population before the arrival 
of the mutant, we can derive the invasion probability (see the proof in Section \ref{proof:proba_ext})
\begin{prop} \label{prop:proba_ext}
Assume that there are only $A$ individuals before the arrival of the mutant ($\proportion_A=1$). In this case, $$q_A=\frac{2}{2+\beta_1}$$
and
$$q_a= \frac{1}{2-\beta_2} \left( \frac{6-\beta_1\beta_2+4\beta_1-\beta_2}{2+\beta_1}-\sqrt{\left(\frac{6-\beta_1\beta_2+4\beta_1-\beta_2}{2+\beta_1}\right)^2-4(2-\beta_2)}\right).$$
\end{prop}
Results obtained with the help of the software Mathematica show a complex dependency with respect to parameters.
We performed numerical simulations of the extinction probabilities $(q_A,q_a)$ using Newton approximation scheme starting from $(0,0)$. We computed the values of $q_A$ as a function of $\rho_A$ for different values of $\beta_1$ and $\beta_2$. Using the symmetry of our model, we have that $q_a(\rho_A)=q_A(1-\rho_A)$.
We observe on Figure \ref{fig_extinction} that $q_A$ is a continuous function of $\rho_A$ but that it is not differentiable near criticality.
\begin{figure}[h!]
\includegraphics[scale=0.5]{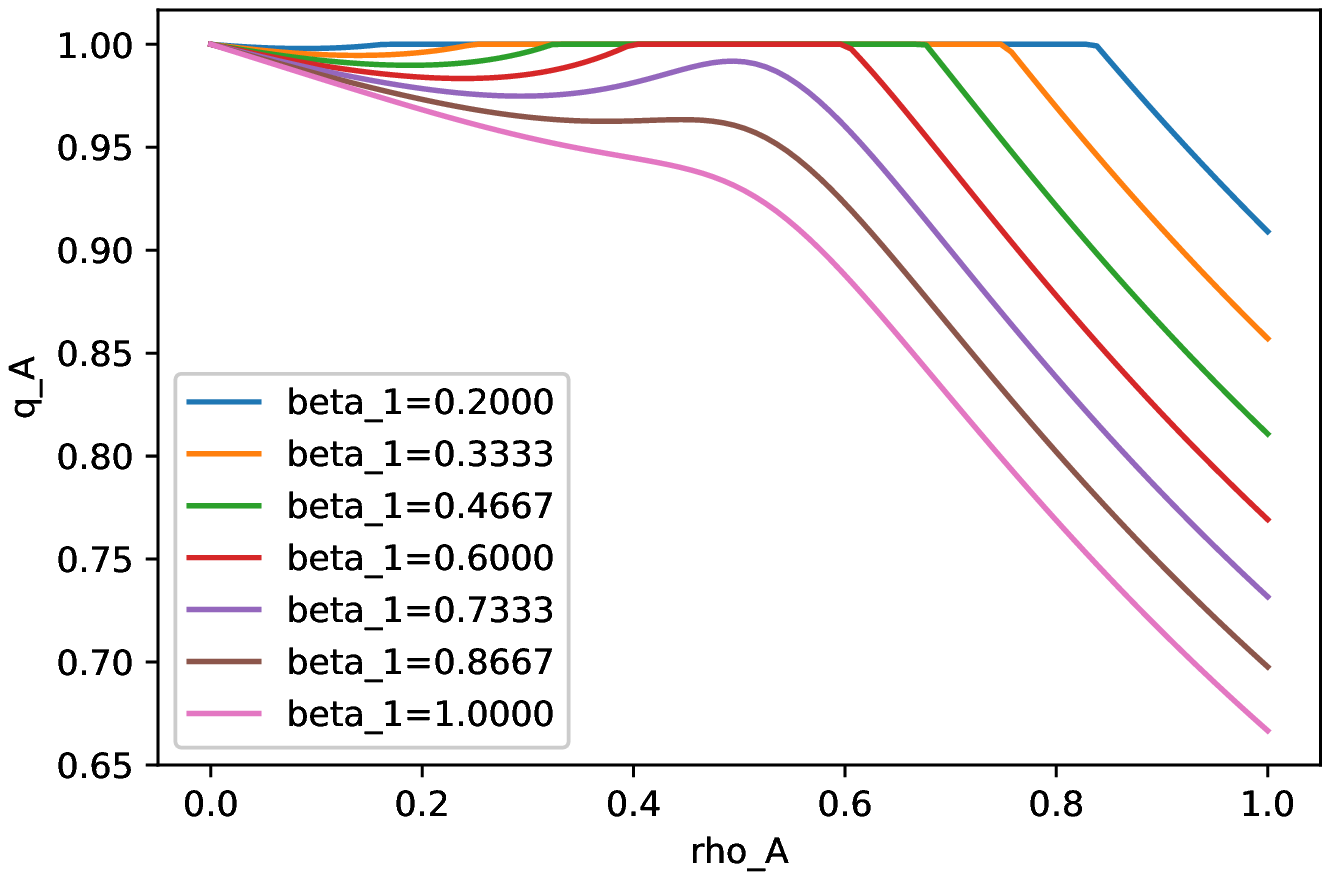}\includegraphics[scale=0.5]{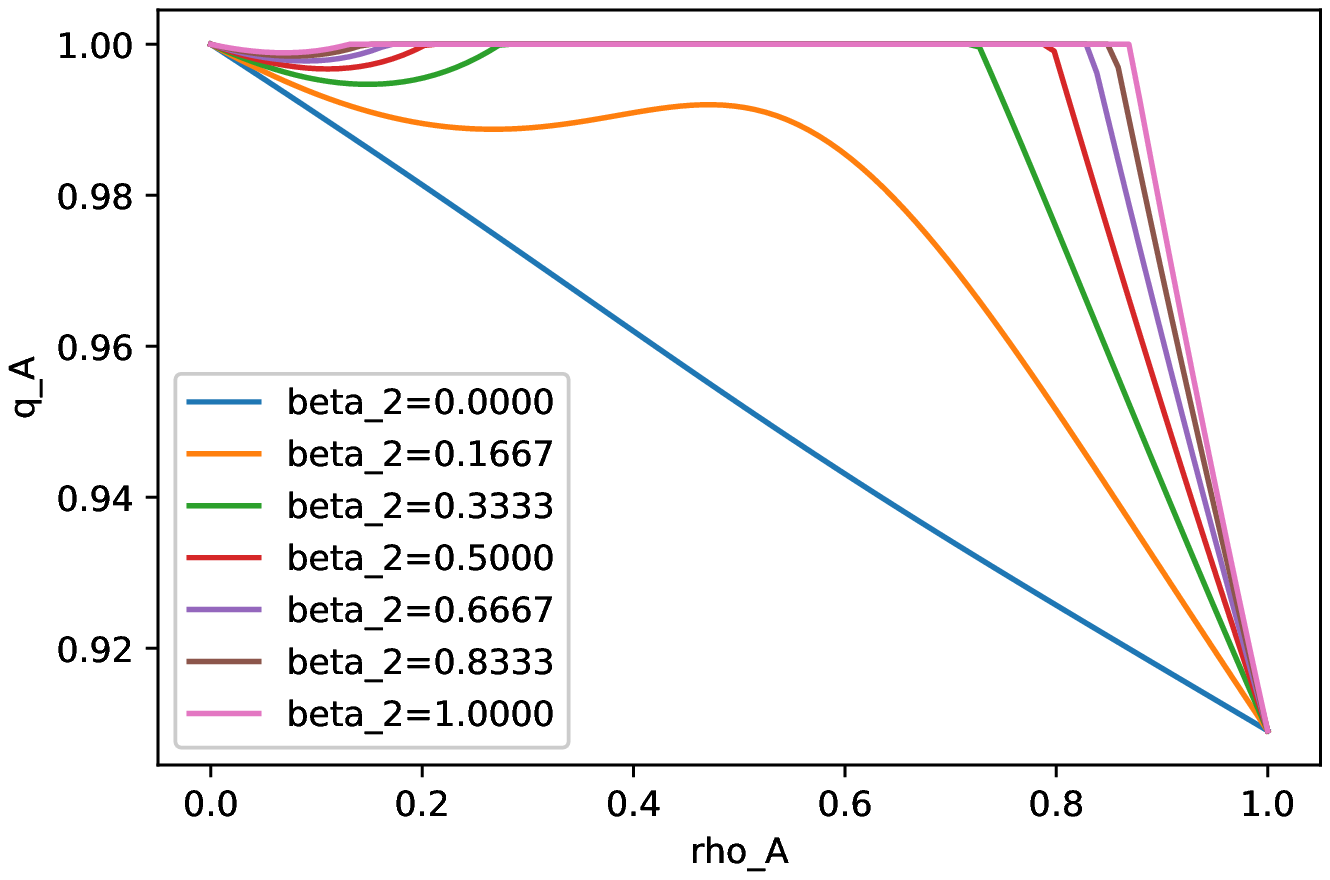}
\caption{\label{fig_extinction} Values of $q_A$ as a function of $\rho_A$ for different values of $\beta_1$ and $\beta_2$. On the left, $\beta_2$ is fixed to $0.7$ and $\beta_1$ varies. On the right $\beta_1$ is fixed to $0.2$ and $\beta_1$ varies. In both cases $b=1$.}
\end{figure}

\begin{rem}
We assumed that the initial population state is close to the equilibrium state of the population when all individuals 
mate uniformly at random, because any neighbourhood of such an equilibrium is reached 
within a finite time as soon as the initial population size is of order $K$. 
We thus could relax this assumption and only assume that the $p$-population size is of order $K$ and $N_{Ap}(0)>N_{ap}(0)$.
This would however require more complex notations.
\end{rem}

The rest of the paper is devoted to the proof of Theorem \ref{theo_main}.
Notice that for the sake of readability, we will not indicate anymore the dependency of the rescaled process $\mathbf{Z}^K$ on $K$ and will instead 
write $\mathbf{Z}$.

\section{Proof of Theorem \ref{theo_main}}

\subsection{Probability and time of the mutant invasion}\label{subsectionInvasion}

The first step of the proof of Theorem \ref{theo_main} consists in studying the population dynamics when a mutant of type $P$ appears in a well-established population of types $ap$ and $Ap$. We would like to know under which conditions on the parameters the mutant population may invade the resident population and what is the probability that the invasion happens.

We will show in particular that when the mutant appears and as long as the mutant population size is negligible with respect to the carrying capacity $K$, its dynamics is close to the dynamics of the process $\bar{\Nbf}$, which has been introduced in Section~\ref{sec:intro}. 
Next, as long as the $P$-population size is small compared with $K$, that is, as long as its dynamics is close to the one of $\bar\Nbf$, we prove that the $p$-population size and the proportion of $A$-individuals in the $p$-population will not vary considerably from their initial values. This part of the proof is more technical in our setting since the equilibria of the resident population are non hyperbolic.

\smallskip 
In order to state rigorously these results, let us recall definition \eqref{defTNT0} and introduce two more stopping times. 
The first one gives the first time when the genetic proportions in the $p$-population deviate considerably from their starting values: for any $\eps>0$,
\begin{equation}\label{defUeps}
 U_\eps:= \inf \left\{t\geq 0, \left|\frac{N_{Ap}(t)}{N_p(t)}-\frac{N_{Ap}(0)}{N_p(0)}\right|>\eps \right\}.
\end{equation}
The second one concerns the total $p$-population size: for any $\eps>0$,
\begin{equation}\label{defReps}
 R_\eps:= \inf \left\{t\geq 0, \left|\frac{N_{p}(t)}{K}-\frac{b-d}{c}\right|>\eps  \right\}.
\end{equation} 
Note that these stopping times depend on the scaling parameter $K$. However, 
to avoid cumbersome notations, we drop the $K$ dependency.

We recall that $q_\alpha$ is the extinction probability of this process and that $\lambda$ is the principal 
eigenvalue of the matrix $J$ defined in \eqref{matrix}, which can be rewritten
\begin{equation}\label{matrixbis} 
J =
 \frac{b}{2} \left( \begin{matrix}
\proportion_A\beta_1-(1-\proportion_A)\left(\frac{\beta_2}{2}+1 \right) & \proportion_A\left(1-\frac{\beta_2}{2}\right) \\
(1-\proportion_A)\left(1-\frac{\beta_2}{2}\right)           & (1-\proportion_A)\beta_1-a_1\left(\frac{\beta_2}{2}+1\right) 
\end{matrix} \right).
 \end{equation}

The main result along the route of proving Theorem \ref{theo_main} can now be stated. It ensures that
the probability that a mutant $P$ generates a $P$-population whose size reaches the order $K$
is close to $1-q_{\alpha}$ (which is the survival probability of the process $\bar{\bf{N}}$ starting 
from an $\alpha$-individual and has been defined in \eqref{def_probaq}), whereas its probability of extinction is close to $q_{\alpha}$. 
Moreover, the invasion or extinction of the mutant population occurs before the resident population size deviates
substantially from its equilibrium, and the time of invasion is approximately $\log(K)/\lambda$.
\begin{prop} \label{prop_equiv_inv}
Let $\alpha$ be in $\mathfrak{A}$ and assume that the initial condition satisfies  
$$ N_P(0)=N_{\alpha P}(0)=1, \quad \quad N_p(0) = \left\lfloor  \frac{b-d}{c} K \right\rfloor\quad \text{and} \quad \proportion_A=\frac{N_{Ap}(0)}{N_p(0)},  $$
and moreover that
\begin{equation} \label{conddetmatrice}
 \lambda\neq 0,
\end{equation}
where $\lambda$ is the principal eigenvalue of matrix~\eqref{matrix}.
There exist a function $\eta$ going to $0$ at $0$ and 
a positive constant $\A_0$ such that for any $\xi\in \{1/2,1\}$,
 $$ \limsup_{K \to \infty} \left| \P \left( T^P_{\eps^\xi } < T^P_0 \wedge R_{\mathcal{A}_0 \eps} 
 \wedge U_{ \eps^{1/6}}, \left| \frac{T_{\eps^\xi }^P}{\ln K}- \frac{1}{\lambda} \right|\leq \eta(\eps) \Big| {\bf{N}}_P(0)={\bf e}_{\alpha} \right)
 - (1- q_{\alpha})  \right| =o_\eps(1), $$
 and
\begin{equation} \label{eq2prop}
  \limsup_{K \to \infty} \left| \P \left( T^P_0  < T^P_{\eps^\xi } \wedge R_{\mathcal{A}_0  \eps} 
 \wedge U_{ \eps^{1/6}} \big| {\bf{N}}_P(0)={\bf e}_{\alpha} \right)
 - q_{\alpha}  \right| =o_\eps(1),
 \end{equation}
where by convention, $o_\eps(1)$ goes to $0$ when $\eps$ goes to $0$.
 \end{prop}

\begin{rem}
This proposition accounts for two opposite behaviours of the mutant process. Indeed Assumption \eqref{conddetmatrice} ensures that  either the process 
$\bar{\mathbf{N}}$ is supercritical ($\lambda>0$ under condition \eqref{eq:noStabNoMig}) and $q_{\alpha} \in (0,1)$, or the process 
$\bar{\mathbf{N}}$ is subcritical and $q_{\alpha}=1$.
\end{rem}

The end of this section will be devoted to the proof of Proposition \ref{prop_equiv_inv}, which will be divided into three steps.

\subsubsection{Control of the proportions in the resident population}
We will first prove that the proportions in the resident population do not vary substantially before the mutant population goes extinct or invades. 
More precisely, we have the following lemma.
\begin{lem} \label{lemapprox}
Suppose that the assumptions of Proposition \ref{prop_equiv_inv} hold. For any $\mathcal{A}>0$, there exist $\eps_0$ such that for any $\xi\in\{1/2,1\}$ and $\eps \leq \eps_0$,
$$ \limsup_{K \to \infty} \P \left(  U_{\eps^{1/6}} 
< R_{\mathcal{A}\eps} \wedge T^P_{\eps^\xi }\wedge T^P_0 \right)
\leq C(\mathcal{A},\xi) \eps^{1/12}, $$
where $C(\mathcal{A},\xi)$ is a positive constant.
\end{lem}

\begin{proof}
The statement of Lemma \ref{lemapprox} is a direct consequence of the following inequality:
\begin{equation}\label{eq:min1} 
\limsup_{K \to \infty} \P \left( \sup_{t \leq U_{\eps^{1/8}} \wedge R_{\mathcal{A}\eps} 
\wedge T^P_{\eps^\xi }\wedge T^P_0} \left| \frac{N_{Ap}(t)}{N_p(t)} -
\frac{N_{Ap}(0)}{N_p(0)} \right|> \eps^{1/6} \right) \leq C \eps^{1/12}.
\end{equation}  
To prove \eqref{eq:min1}, we decompose the process $N_{Ap}/N_p$ as the sum of a square integrable martingale $M_p$
and of a finite variation process $V_p$ (see \eqref{Mp} and \eqref{Ap} for their expressions). Using such a decomposition and introducing for the sake of readability
the notation
\begin{equation} \label{def_tau_eps} \tau_\eps:=U_{\eps^{1/8}} \wedge R_{\A\eps} \wedge T^P_{\eps^\xi }\wedge T^P_0,\end{equation}
we find that for $\eps$ small enough,
\begin{equation} \label{ineq_MA}
\begin{aligned}
 \P \bigg( \sup_{t \leq \tau_\eps}& \left| \frac{N_{Ap}(t)}{N_p(t)} -
\frac{N_{Ap}(0)}{N_p(0)} \right|> \eps^{1/6} \bigg)\\
&\leq \P \left(  \sup_{t \leq \tau_\eps} \left| M_p(t) \right|>\frac{\eps^{1/6}}{2} \right) +
 \P \left(  
 \sup_{t \leq \tau_\eps} \left| V_p(t)\right|  > 
\frac{\eps^{1/6}}{2} \right)\\
&\leq \frac{2}{\eps^{1/6}}  \E\left[ \left|M_p(\tau_\eps)\right| \right] +
 \frac{\sqrt{2}}{\eps^{1/12}} \E \left[ \sqrt{\sup_{t \leq \tau_\eps} 
 \left| V_p(t)\right|}\right] \\
&\leq \frac{2}{\eps^{1/6} } \left( \sqrt{\E\left[ M_p^2(\tau_\eps) \right]} +
\sqrt{\E \left[ \sup_{t \leq \tau_\eps} \left| V_p(t)\right|\right] } \right),
\end{aligned}
\end{equation}
where we applied Doob maximal, Markov, Cauchy-Schwarz and Jensen inequalities.

Hence, it remains to bound the two last expectations of \eqref{ineq_MA}.
In the vein of Fournier and M\'el\'eard \cite{fournier2004microscopic} we represent the population
process in terms of Poisson measures.

Let $(Q^{(\rho)}_{\alpha p}(ds,d\theta),\alpha \in \mathfrak{A}, \rho \in \{b,d\})$ be four 
independent Poisson random measures on $\R^2_+$ with
intensity $dsd\theta$ representing respectively the birth and death events of $Ap$ and $ap$ individuals. 
That is, for any $\alpha\in \mathfrak{A}$, the $p$-population size processes can be written
\begin{equation}\label{def:PPoisson}
N_{\alpha p }(t)=N_{\alpha p}(0)+\int_0^t\int_{\R^+} \left( \mathbf{1}_{\{\theta\leq b_{\alpha p}(\Nbf(s-))\}} Q^{(b)}_{\alpha p}(ds,d\theta) - 
\mathbf{1}_{\{\theta\leq d_{\alpha p}(\Nbf(s-))\}} Q^{(d)}_{\alpha p}(ds,d\theta)\right)
\end{equation}
where the quantities $b_{\alpha p}$ and $d_{\alpha p}$ have been defined in \eqref{eq:mort} and \eqref{eq:naissAP}.

Let us also denote by $\tilde{Q}^{(\varrho)}_{\alpha p}(ds,d\theta):=Q^{(\varrho)}_{\alpha p}(ds,d\theta)-dsd\theta$ the associated compensated measure, for any $ \varrho \in \{b,d\}, \alpha \in \mathfrak{A}$. 
From~\eqref{def:PPoisson}, we find, for $t\geq 0$,
\[
\frac{N_{Ap}(t)}{N_p(t)} =  \frac{N_{Ap}(0)}{N_p(0)}+M_p(t)+V_p(t),
\]
with $M_p$ and $V_p$ such that:
\begin{align}\label{Mp}
M_p(t) = & \int_{0}^{t} \int_{\R_+} \mathbf{1}_{\{ \theta \leq b_{Ap}(\Nbf(s-))\}}
\frac{N_{ap}(s-)}{N_p(s-)(N_p(s-)+1)}\tilde{Q}_{Ap}^{(b)}(ds,d\theta) 
 \\&- \int_{0}^{t} \int_{\R_+} \mathbf{1}_{\{\theta \leq d_{Ap}(  \Nbf(s-))\}}\frac{N_{ap}(s-)}{N_p(s-)(N_p(s-)-1)}\tilde{Q}_{Ap}^{(d)}(ds,d\theta) \nonumber
 \\&- \int_{0}^{t} \int_{\R_+} \mathbf{1}_{\{\theta \leq b_{ap}( \Nbf(s-))\}}\frac{N_{Ap}(s-)}{N_p(s-)(N_p(s-)+1)}\tilde{Q}_{ap}^{(b)}(ds,d\theta) \nonumber
 \\&+ \int_{0}^{t} \int_{\R_+} \mathbf{1}_{\{\theta \leq d_{ap}(\Nbf(s-)) \}}\frac{N_{Ap}(s-)}{N_p(s-)(N_p(s-)-1)}\tilde{Q}_{ap}^{(d)}(ds,d\theta),  \nonumber
\end{align}
and
\begin{align} \label{Ap}
V_p(t) = & \int_{0}^{t} \Big\{ b_{Ap}(\Nbf(s))\frac{N_{ap}(s)}{N_p(s)(N_p(s)+1)}
-d_{Ap}(\Nbf(s))\frac{N_{ap}(s)}{N_p(s)(N_p(s)-1)}\nonumber
 \\&-  b_{ap}(\Nbf(s))\frac{N_{Ap}(s)}{N_p(s)(N_p(s)+1)}
 +  d_{ap}(\Nbf(s))\frac{N_{Ap}(s)}{N_p(s)(N_p(s)-1)}\Big\}ds \nonumber \\
 =&\int_{0}^{t} \Big\{ b_{Ap}(\Nbf(s))N_{ap}(s)-b_{ap}(\Nbf(s))N_{Ap}(s)\Big\}\frac{ds}{N_p(s)(N_p(s)+1)} .
 \end{align}
Using Equation \eqref{eq:naissAP}, we obtain the existence of a finite constant $C$ such that 
$$\left|  b_{Ap}(N)N_{ap}-b_{ap}(N)N_{Ap}\right|\leq C \frac{N_PN^2_p}{N_P+N_p}.$$
Hence, 
\begin{equation}\label{eq:majAp} 
 \sup_{t\leq \tau_\eps} |V_p(t)|\leq C \int_0^{\tau_\eps} \frac{N_P(s)}{N_P(s)+N_p(s)}ds.
 \end{equation}
This will help us bounding the last term of inequality~\eqref{ineq_MA}. 
On the other hand, to deal with the penultimate term in \eqref{ineq_MA}, we use the quadratic variation of the martingale $M_p$ which equals
\begin{align}\label{def_crochet}
  \langle M_p\rangle_{\tau_\eps}=&
\int_{0}^{\tau_\eps} b_{Ap}(\Nbf(s))\frac{N^2_{ap}(s)}{N^2_p(s)(N_p(s)+1)^2}ds
+ \int_{0}^{\tau_\eps} d_{Ap}(\Nbf(s))\frac{N^2_{ap}(s)}{N^2_p(s)(N_p(s)-1)^2}ds
 \\&+ \int_{0}^{\tau_\eps} b_{ap}(\Nbf(s))\frac{N^2_{Ap}(s)}{N^2_p(s)(N_p(s)+1)^2}ds
+ \int_{0}^{\tau_\eps} d_{ap}(\Nbf(s))\frac{N^2_{Ap}(s)}{N^2_p(s)(N_p(s)-1)^2}ds. \nonumber \\
=&
\int_{0}^{\tau_\eps} \left(\frac{b_{Ap}(\Nbf(s))N^2_{ap}(s)+b_{ap}(\Nbf(s))N^2_{Ap}(s)}{N^2_p(s)(N_p(s)+1)^2}
+ \frac{d_{Ap}(\Nbf(s))N^2_{ap}(s)+d_{ap}(\Nbf(s))N^2_{Ap}(s)}{N^2_p(s)(N_p(s)-1)^2}\right)ds \nonumber
\end{align}
To handle the first term, let us remark that  $b_{Ap}(\Nbf)$ and $b_{ap}(\Nbf)$ can be bounded from above by 
$\tilde C N_p$ for a positive constant $\tilde C$. Therefore 
$$ \frac{b_{Ap}(\Nbf)N^2_{ap}+b_{ap}(\Nbf)N^2_{Ap}}{N^2_p(N_p+1)^2}\le \frac{\tilde C}{N_p}.
$$ For the second term we have 
$$\frac{d_{Ap}(\Nbf)N^2_{ap}+d_{ap}(\Nbf)N^2_{Ap}}{N^2_p(N_p-1)^2}\leq \frac{(d+cN/K)N_{ap}N_{Ap}}{N_p(N_p-1)^2}.$$
Since, for $N_p\geq 2$, 
$$ \frac{N_{a p}N_{A p}}{N_p(N_p-1)^2}\leq \frac{4}{N_p}, $$
we obtain that, if $C$ and $K$ are sufficiently large,
\begin{equation}\label{eq:majMpbracket}
\begin{aligned}
   \langle M_p\rangle_{\tau_\eps} &\leq \int_0^{\tau_\eps} \frac{4}{N_p(s)}\left[\tilde C+d+\frac{c}{K}(N_P(s)+N_p(s))\right] ds  
 \\&   \leq C
 \int_{0}^{\tau_\eps} \frac{1}{N_p(s)}ds
 \leq C
 \int_{0}^{\tau_\eps} \frac{N_P(s)}{N_p(s)}ds.
\end{aligned}
\end{equation}
From \eqref{eq:majAp} and \eqref{eq:majMpbracket}, we get that there exists a finite $C$ such that
\begin{equation}\label{ineq_int}
 \sqrt{\E\left[ M_p^2(\tau_\eps) \right] }+
\sqrt{\E \left[ \sup_{t \leq \tau_\eps} \left| V_p(t)\right|\right]}  \leq
\sqrt{\frac{C}{K} \E \left[ \int_0^{\tau_\eps} N_P(s)ds\right]}.
\end{equation}
In view of~\eqref{ineq_MA}, the problem is thus reduced to show the following property:
$$ \eps^{-1/6} \sqrt{\frac{C}{K} \E \left[ \int_0^{\tau_\eps} N_P(s)ds\right]}\leq C' \eps^{1/12}, $$
for a finite $C'$, or equivalently,
\begin{equation} \label{toshow} 
\E \left[ \int_0^{\tau_\eps} N_P(s)ds\right]\leq C'K \eps^{1/2}. 
\end{equation}
To this aim, we will prove that there exist two real numbers $\gamma_1$ and $\gamma_2$ such that the function $f$ 
on $\N^4$ defined by
\begin{equation} \label{def_f}
f(\textbf{N}): = \gamma_1 N_{AP}+\gamma_2 N_{aP},
\end{equation}
satisfies that there exists $\varepsilon$ sufficiently small such that for any 
$t\leq \tau_\eps$ (recall equation \eqref{def_tau_eps}), 
\begin{equation}
\label{eq:clefin}
 \mathcal{L}f(\Nbf(t))\geq N_P(t).
\end{equation}
Here $\mathcal{L}$ is the infinitesimal generator of $\Nbf$.
Indeed, if \eqref{eq:clefin} holds, it will imply that
\begin{equation}\label{eq_intZ}
\begin{aligned}
 \frac{C}{K} \E \left[ \int_0^{\tau_\eps} N_P(s)ds\right] 
 &\leq   \frac{C}{K} \E \left[ \int_0^{\tau_\eps} \mathcal{L}f(\Nbf(s))ds\right] \\
 &= \frac{C}{K} \E \left[ f(\Nbf(\tau_\eps))- f(\Nbf(0)) \right] \\
 &\leq \frac{C}{K} \left( \max\{\gamma_1,\gamma_2\} \eps^\xi K -\min\{\gamma_1,\gamma_2\} \right),
\end{aligned}
\end{equation}
which is sufficient to obtain \eqref{toshow}, whatever the signs of $\gamma_1$ and $\gamma_2$.\\

The last step of the proof consists in proving the existence of $\gamma_1$ and $\gamma_2$ satisfying \eqref{def_f} and \eqref{eq:clefin}.
Let us now apply the infinitesimal generator of $\mathbf{N}$ to the function $f$ defined in \eqref{def_f}:
\begin{equation}\label{eq:generator}
\begin{aligned}
\mathcal{L}f(\Nbf(t))&=\gamma_1 \left[ b_{AP}(\Nbf)-d_{AP}(\Nbf)\right]
+\gamma_2 \left[ b_{aP}(\Nbf)-d_{aP}(\Nbf) \right]\\
&= N_{AP}(t) \left[\gamma_1 \left(\beta^{(P)}_{AA}(t)-\delta(t) \right)+\gamma_2 \beta^{(P)}_{Aa}(t)\right]\\
&\qquad+ 
N_{aP}(t) \left[\gamma_1 \beta^{(P)}_{aA}(t)+\gamma_2  \left(\beta^{(P)}_{aa}(t)-\delta(t) \right) \right],
\end{aligned}
\end{equation}
where 
\begin{equation}
\label{def:delta}\delta(t)=d+cN(t)/K
\end{equation} and for $\alpha \in \mathfrak{A}$, 
\begin{equation}\label{def_betaAA}
\beta^{(P)}_{\alpha \alpha}(t)= b\left(1+\frac{\beta_1}{2}\frac{2N_{\alpha P}(t)+N_{\alpha p}(t)}{N_P(t)+N_p(t)}-
\frac{\beta_2}{4} \frac{4N_{\bar\alpha P}(t)+N_{\bar\alpha p}(t)}{N_P(t)+N_p(t)}  -\frac{1}{2}\frac{N_{\bar{\alpha}p}(t)}{N_P(t)+N_p(t)} \right)
\end{equation}
and
\begin{equation} \label{def_betaAa}
\beta^{(P)}_{\alpha \bar{\alpha}}(t)=  \frac{b}{2}\left(1-\frac{\beta_2}{2}\right) \frac{N_{\bar\alpha p}(t)}{N_P(t)+N_p(t)}.
\end{equation}

Then we see that to obtain \eqref{eq:clefin} it is enough to choose $\Gamma=(\gamma_1,\gamma_2)$ such that $\forall t\le\tau_\varepsilon$
\begin{equation*}
 J^{(P)}(t)\Gamma^T:=
  \begin{pmatrix}
  \beta^{(P)}_{AA}(t)-\delta(t) &\beta^{(P)}_{Aa}(t) \\
  \beta^{(P)}_{aA}(t) & \beta^{(P)}_{aa}(t)-\delta(t)
 \end{pmatrix}
 \begin{pmatrix}
  \gamma_1\\\gamma_2
 \end{pmatrix} > \begin{pmatrix}
  1\\1
 \end{pmatrix}
\end{equation*}
where the inequality is applied to each coordinate.

Note that $J^{(P)}(t)$ is not easy to study. We will thus approximate this matrix by a simpler one as 
soon as $t\leq \tau_\eps$. More precisely, we will prove that there exists a constant $C$ such that for every $t\le \tau_\varepsilon$, 
\ben\label{eq:approxMatrix}
\left(\lvert J^{(P)}(t)-J\rvert\right)_{ij} \le C\varepsilon^{1/8}, \quad i,j\in\{1,2\}
\een
where the matrix $J$ given in \eqref{matrix} is the infinitesimal generator of the branching process 
$\bar{\Nbf}$ (defined in \eqref{defbarN}) which approximates $(N_{AP},N_{aP})$ near the equilibrium of the resident population. \\

First, as $t\leq \tau_\eps \leq T^P_{\eps^\xi K} \wedge R_{\mathcal{A}\eps}$, we have
\begin{equation}\label{ineg1}
\left| b-d-c\frac{N_P+N_p}{K} \right|\leq c\left| \frac{b-d}{c}-\frac{N_p}{K}\right|+ c\left|\frac{N_P}{K}\right|\leq 
c \mathcal{A} \eps+c\eps^\xi.
\end{equation}
Secondly, using also that $t\leq U_{\eps^{1/8}}$ and recalling that $\proportion_A=N_{Ap}(0)/N_p(0)$, 
we find that for $\eps$ small enough,
\begin{equation}\label{ineg2}
\begin{aligned}
 \left| \frac{2N_{A P}+N_{A p}}{N_P+N_p}-\proportion_A \right| &\leq \frac{2N_{A P}}{N_P+N_p}+\proportion_A\left|\frac{N_{p}}{N_P+N_p}-1\right|+ \frac{N_{p}}{N_P+N_p} \left|\frac{N_{A p}}{N_p}-\proportion_A\right|\\
 & \leq (2+\proportion_A) \frac{N_{P}}{N_P+N_p} + \eps^{1/8}\\
 &\leq 3 \frac{\eps^\xi c }{b-d-c \mathcal{A} \eps} +\eps^{1/8} \leq C_1 \eps^{1/8},
\end{aligned}
\end{equation}
where $C_1$ is a finite constant and $\xi\in \{1/2,1\}$.
Similarly, we prove that, if $C_1$ is sufficiently large,
\begin{equation}
 \label{ineg3}
  \left| \frac{4N_{a P}+N_{a p}}{N}-(1-\proportion_A) \right|\leq C_1\eps^{1/8} \quad \text{and}
  \quad \left| \frac{N_{a p}}{N_P+N_p}-(1-\proportion_A) \right|\leq C_1\eps^{1/8}.
\end{equation}
Using \eqref{ineg1}, \eqref{ineg2} and \eqref{ineg3}, we can find a positive constant $C_2$ such that
\begin{equation*}
 \begin{aligned}
  \Big| \beta^{(P)}_{AA}-\delta  - (\bar{\beta}_{AP}-b) \Big| & \leq   \left| b-d-c\frac{N_P+N_p}{K}\right| + \frac{b\beta_1}{2}\left| \frac{2N_{A P}+N_{A p}}{N_P+N_p}-\proportion_A \right| \\
  & \quad  +\frac{\beta_2}{4} \left| \frac{4N_{a P}+N_{a p}}{N_P+N_p}-(1-\proportion_A) \right| +\frac{1}{2}\left| \frac{N_{a p}}{N_P+N_p}-(1-\proportion_A) \right|
  \leq C_2 \eps^{1/8},
 \end{aligned}
\end{equation*}
where we recall that $\delta$ has been defined in \eqref{def:delta}, and
\begin{equation*}
 \begin{aligned}
  \left| \beta^{(P)}_{Aa}-\bar{\beta}_{Aa}\right|\leq \frac{b(2-\beta_2)}{4} 
  \left| \frac{N_{a p}}{N_P+N_p}-(1-\proportion_A) \right|\leq C_2 \eps^{1/8}.
 \end{aligned}
\end{equation*}
The last terms in \eqref{eq:generator} can be bounded using similar computations, which yields \eqref{eq:approxMatrix}.\\

Let us finally choose $(\gamma_1,\gamma_2)$. 
Recall the definition of $J$ in \eqref{matrix}. In particular, the matrix $J+b Id$ has positive coefficients and we can apply Perron-Frobenius Theorem: $J+b Id$ possesses a positive right eigenvector $\tilde{\Gamma}=(\tilde{\gamma}_1,\tilde{\gamma}_2)$ associated to the positive principal eigenvalue  $\lambda+b$ and thus
$ (J+b Id)\tilde{\Gamma}^T=(\lambda+b)\tilde{\Gamma}^T, $ and $$J\tilde{\Gamma}^T=\lambda\tilde{\Gamma}^T.$$
Since both coordinates of $\tilde{\Gamma}$ are positive and $\lambda\neq0$ , we can define $\Gamma=2\tilde{\Gamma}(\lambda(\tilde{\gamma}_1\wedge\tilde{\gamma_2}))^{-1}=(\gamma_1,\gamma_2)$. It is 
solution to
\begin{equation*}
 J\Gamma^T=\lambda\Gamma^T\quad \text{where}\quad \lambda\gamma_i\ge2,\quad  \forall i\in\{1,2\}.
\end{equation*}
Combining with \eqref{eq:approxMatrix}, we deduce
\begin{align*}
  \left| \gamma_1 \left(\beta^{(P)}_{AA}(t)-\delta  \right)+\gamma_2 \beta^{(P)}_{Aa}(t) -\lambda\gamma_1\right|
&{=} \left| \gamma_1 \left[\beta^{(P)}_{AA}(t)-\delta  - (\bar{\beta}_{AA}-b)  \right]+\gamma_2 
  \left[\beta^{(P)}_{Aa}(t)-\bar{\beta}_{Aa}\right]\right|\\
& \leq \left(|\gamma_1|+|\gamma_2|\right) C_2 \eps^{1/8}.
\end{align*}
Finally, as $\lambda \gamma_1\geq 2$, if $\eps$ is sufficiently small, for any  $t\leq \tau_\eps$, 
$$\gamma_1 \left(\beta^{(P)}_{AA}(t)-\delta(t)  \right)+\gamma_2 \beta^{(P)}_{Aa}(t) \geq 1,$$ 
which leads to \eqref{eq:clefin} and ends the proof of Lemma~\ref{lemapprox} using similar computations for the second term.
\end{proof}

\subsubsection{Control of the resident population size}

Lemma \ref{lemapprox} ensures that the proportions of types $A$ and $a$ in the $p$-population 
stay almost constant during the time interval under consideration. We now prove that it is also the case for the total 
$p$-population size.

\begin{lem} \label{lemRepsplusgrand}
Under the assumptions of Proposition \ref{prop_equiv_inv} there exist two finite constants $\A_0$ and $\eps_0$ 
such that for any $\xi \in \{1/2,1\}$ and $\eps \leq \eps_0$,
$$ \limsup_{K \to \infty} \P \left( R_{\mathcal{A}_0\eps} < U_{\eps^{1/6}} \wedge T^P_{\eps^\xi }\wedge T^P_0  \right)= 0. $$
\end{lem}

\begin{proof}
Recall that $\Zbf=\Nbf/K$. As long as $t\leq T^P_0 \wedge T^P_{\eps^\xi }$, we couple the process $Z_{p}$, which describes the total 
$p$-population size dynamics, 
with two birth and death processes,
$Z^{1}_p$ and $Z^{2}_p$ such that 
$$
Z^{1}_p(t) \leq Z_p(t) \leq Z^{2}_p(t), \quad\text{a.s.} \quad \forall t \leq T^P_0\wedge T^P_{\eps^\xi }.
$$
To this aim, we use bounds on the birth and death rates of $Z_p$. 
Once again, everything depends on $K$, but for the sake of readability, we drop the $K$ dependency. 
Since $\beta_2\leq 1$, processes $Z^{1}_p$ and $Z^{2}_p$
may be chosen with
the following birth and death rates
\begin{center}
\begin{tabular}{lccccl}
$Z^{1}_p:$&$\frac{i}{K}$ & $\to$ & $\frac{i+1}{K}$ & at rate & $K\left(b\frac{i}{K} -b\eps^\xi\right)$\\
&$\frac{i}{K}$ & $\to$ & $\frac{i-1}{K}$ & at rate & $K\frac{i}{K}\left( d+c\eps^\xi+ c\frac{i}{K}\right)$
\end{tabular}
\end{center}
and
\begin{center}
\begin{tabular}{lccccl}
 $Z^{2}_p:$&$\frac{i}{K}$ & $\to$ & $\frac{i+1}{K}$ & at rate & $K\left(b\frac{i}{K} + b\eps^\xi 
 \frac{(\beta_1+1)}{2}\right)$\\
&$\frac{i}{K}$ & $\to$ & $\frac{i-1}{K}$ & at rate & $K\frac{i}{K}\left(d+c\frac{i}{K} \right)$.
\end{tabular}
\end{center}

We will first prove that processes $Z^1_p$ and $Z^2_p$ stay close to the value $\zeta:=(b-d)/c$ for at least an exponential (in $K$) time with 
a probability close to one when $K$ is large. To this aim, we will study the following stopping times
\begin{equation} \label{deftpssortie}
R^i_{\eta}:=\inf \left\{ t\geq 0, Z^{i}_p \not\in [\zeta-\eta,\zeta+\eta] \right\},
\end{equation}
for $\eta>0$ and $i \in \{\emptyset,1,2\}$ (by convention $Z^{\emptyset}_p=Z_p$).

Let us first consider the process $Z^{1}_p$.
When $K$ is large and according to \cite[Chapter 11, Theorem 2.1 p. 456]{Ethier-Kurtz}, the dynamics of $Z^{1}_p$ is close to the dynamics of the unique solution to
\begin{equation} 
\label{ZK1largepop}
\frac{d}{dt}z=z(b-d-c\eps^\xi-cz)-b\eps^\xi.
\end{equation}
The differential equation \eqref{ZK1largepop} admits two positive equilibria:
$$\zeta^{1,\pm}(\eps):=\frac{b-d-c\eps^\xi \pm\sqrt{(b-d-c\eps^\xi)^2-4b\eps^\xi c}}{2c}.$$
A direct analysis of the sign of $z(b-d-c\eps^\xi-cz)-b\eps^\xi$ shows that for any fixed $\eps >0$, any solution
with initial condition on $]\zeta^{1,-}(\eps),+\infty[$, converges to the stable equilibrium 
$\zeta^{1,+}(\eps)$ when $t$ goes to infinity. 
Moreover, using 
\begin{equation} \label{racine}  
(\sqrt{a}+\sqrt{b})(\sqrt{a}-\sqrt{b})=a-b
\end{equation}
 yields, if $\eps\leq \eps_0$, for any $\eps_0$ sufficiently small such that $\zeta^{1,+}(\eps_0)\geq (b-d)/2c$,
$$
\zeta^{1,-}(\eps) = \frac{b\eps^\xi}{c\zeta^{1,+}(\eps)}\leq \frac{2b}{b-d}\eps^\xi\leq 2\eps_0^\xi.
$$
Then, we can find two constants $\A_0$ and $\eps_0$ such that, for any $\eps\leq \eps_0$, 
$$|\zeta^{1,+}(\eps)-\zeta|\leq (\A_0-1)\eps^\xi \quad \text{and} \quad 2\eps_0^\xi \not\in [\zeta-\A_0\eps^\xi,
\zeta+\A_0\eps^\xi].$$
Moreover, using a reasoning similar to the one in the proof of Theorem 3(c) in~\cite{champagnat2006microscopic} (see also Proposition 4.1 in \cite{Coron2018}), we construct a family (over $K$) of Markov jump processes 
$\wt{Z}^{1}_p$ whose transition rates are positive, bounded, Lipschitz and uniformly bounded away from $0$, and for which the following estimate holds 
(Chapter 5 of Freidlin and Wentzell~\cite{freidlin1984random}): there exists $V>0$ such that,
\begin{equation} \label{sortieZK1}
  \P(R^1_{\A_0\eps}>e^{KV})=\P(\tilde R^1_{\A_0\eps}>e^{KV})\underset{K\to+\infty}{\to}1,
\end{equation}
where $\tilde R^1_{\eta}$ is defined similarly as $R^1_{\eta}$ but for the process $\wt{Z}^{1}$.

Using a similar reasoning for process $Z^2_p$ and if $\eps$ and $V$ are small enough and $\A_0$ is large enough, we have that
\begin{equation} \label{sortieZK2}
\P(R^2_{\A_0\eps}> e^{KV})\underset{K\to+\infty}{\to}1.
\end{equation}

Finally, note firstly that $ R_{\mathcal{A}_0\eps} \geq R^1_{\mathcal{A}_0\eps} \wedge R^2_{\mathcal{A}_0\eps}$ on the set 
$\{R_{\mathcal{A}_0\eps} \leq T_0^P \wedge T^P_{\eps^\xi }\}$. In addition with \eqref{sortieZK1} and \eqref{sortieZK2}, we deduce that 
$$ \P(R_{\mathcal{A}_0\eps}\leq e^{KV}, R_{\mathcal{A}_0\eps} \leq T^P_0 \wedge T^P_{\eps^\xi } )\underset{K\to+\infty}{\to} 0. $$
Secondly, for $\eps$ small enough,
\begin{equation*}
 \begin{aligned}
  \P(R_{\mathcal{A}_0\eps}&\leq  T^P_0 \wedge T^P_{\eps^\xi } \wedge U_{\eps^{1/6}} )\\
  & \leq \P(R_{\mathcal{A}_0\eps}\leq e^{KV},  R_{\mathcal{A}_0\eps}\leq T^P_0 \wedge T^P_{\eps^\xi } 
  \wedge U_{\eps^{1/6}})+\P(R_{\mathcal{A}_0\eps}\wedge T^P_0 \wedge T^P_{\eps^\xi } \wedge U_{\eps^{1/6}} 
  \geq e^{KV})\\
  &\leq \P(R_{\mathcal{A}_0\eps}\leq e^{KV}, R_{\mathcal{A}_0\eps} \leq T^P_0 \wedge T^P_{\eps^\xi } )+e^{-KV} 
  \E \left[  R_{\mathcal{A}_0\eps} \wedge T^P_{\eps^\xi }\wedge T^P_0 \wedge U_{\eps^{1/8}} \right].
 \end{aligned}
\end{equation*}
Thirdly, Equation \eqref{eq_intZ} implies that for a finite $C$ and $\eps$ small enough,
\begin{equation*} \label{majesptpsarret} 
\E \left[ R_{\mathcal{A}_0\eps} \wedge T^P_{\eps^\xi }\wedge T^P_0 \wedge U_{\eps^{1/8}} \right] 
\leq \E \left[ \int_0^{U_{\eps^{1/8}} \wedge R_{\mathcal{A}_0\eps} \wedge T^P_{\eps^\xi }\wedge T^P_0} N_P(s)ds\right] \leq
C(\eps^\xi K+1).
\end{equation*}
The three last equations imply the statement of Lemma \ref{lemRepsplusgrand}, which ends its proof.
\end{proof}

\subsubsection{Proof of Proposition \ref{prop_equiv_inv}}

Lemmas \ref{lemapprox} and \ref{lemRepsplusgrand} give us a control on the $p$-population size and the proportions of $A$ and $a$ individuals in this population.
It will allow us to approximate the mutant population size by a bitype branching process at the beginning of the invasion process.
We will assume along the proof that ${N}_P(0)=N_{\alpha P}(0)=1$, with $\alpha \in \mathfrak{A}$, but we drop the conditioning notation for 
the sake of readability.
Combining Lemmas \ref{lemapprox} and \ref{lemRepsplusgrand}, we obtain that the $p$-population size and the genotypic proportions in the $p$-population stay almost constant 
as long as the $P$-mutant population size is small. More precisely, if \eqref{conddetmatrice} is satisfied, 
there exist two constants $\A_0$ and $\eps_0$ such that 
for any $\xi\in \{1/2,1\}$ and $\eps \leq \eps_0$,
\begin{equation} \label{phase_1}
\liminf_{K \to \infty} \P \left( T^P_{\eps^\xi }\wedge T^P_0 <R_{\mathcal{A}_0\eps} 
\wedge U_{\eps^{1/6}} \right)\geq 1-C(\A_0,\xi)\eps^{1/12},
\end{equation}
where $C(\mathcal{A}_0,\xi)$ is a positive constant.
Hence, in what follows, we study the process on the event 
$$
\varSigma_\eps = \left\{T^P_{\eps^\xi }\wedge T^P_0 <R_{\mathcal{A}_0\eps} 
\wedge U_{\eps^{1/6}}\right\},
$$
which has a probability close to $1$. 
On this event, the death or invasion of the mutant population will occur before that the $p$-population deviates
substantially from its initial composition. Thus, we can study the mutant population dynamics by approximating the resident population dynamics with a 
constant dynamics. More precisely, we couple the process $(N_{AP},N_{aP})$ on $\varSigma_\eps$ with 
two multitype birth and death processes $N^{(\eps,-)}$ and $N^{(\eps,+)}$ with values in $\N^2$
such that almost surely, for any 
$t \leq T^P_{\eps^\xi }\wedge T^P_0 \wedge R_{\mathcal{A}_0\eps} \wedge U_{\eps^{1/6}}$ 
and $\alpha \in \mathfrak{A}$,
\begin{equation} \label{couplage1} 
\begin{aligned}
&N_{\alpha}^{(\eps,-)}(t)\leq \bar{N}_\alpha(t)
\leq N_{\alpha}^{(\eps,+)}(t),\\
&N_{\alpha P}^{(\eps,-)}(t)\leq  N_{\alpha P}(t)
\leq N_{\alpha P}^{(\eps,+)}(t). \end{aligned} 
\end{equation}
For $* \in \{+,-\}$, the process $N^{(\eps,*)}$ may be chosen with the rates:
$$\begin{array}{llll}
N^{(\eps,*)} \to N^{(\eps,*)}+ \mathbf{e}_{\alpha} &\text{at rate} &
  \beta_{A \alpha}^{(\eps,*)}N_{A}^{(\eps,*)}
 + \beta_{a \alpha}^{(\eps,*)}N_{a}^{(\eps,*)}\\
  N^{(\eps,*)} \to N^{(\eps,*)}- \mathbf{e}_{\alpha} &\text{at rate} &
  \delta_{\alpha}^{(\eps,*)} N_{\alpha}^{(\eps,*)}
\end{array}$$
where
$$
 \beta_{\alpha \alpha}^{(\eps,+)}= b \left( 1+ \frac{\beta_1}{2} 
\left(\proportion_\alpha+\eps^{1/6} +\frac{2\varepsilon^\xi}{\zeta-\mathcal{A}_0\varepsilon}\right)
- \left( \frac{\beta_2}{4}+ \frac{1}{2} \right) \left(\proportion_{\bar{\alpha}}-\eps^{1/6} \right)\frac{
\zeta  -\mathcal{A}_0\eps}{\zeta +\mathcal{A}_0\eps+ \eps^\xi}  \right)
$$
$$
\beta_{\alpha \bar{\alpha}}^{(\eps,+)}= \frac{b}{2}\left(1-\frac{\beta_2}{2} \right) (\proportion_{\bar{\alpha}}+\eps^{1/6} ),
$$
$$
\delta_{\alpha}^{(\eps,+)}= b -c\mathcal{A}_0\eps, 
$$
$$
 \beta_{\alpha \alpha}^{(\eps,-)}= b \left( 1+ \frac{\beta_1}{2} 
 \frac{(\proportion_\alpha-\eps^{1/6})(\zeta  -\mathcal{A}_0\eps)}{\zeta +\mathcal{A}_0\eps+\varepsilon^\xi}
- \left(\frac{\beta_2}{4}+\frac{1}{2}\right)(\proportion_{\bar{\alpha}}+\eps^{1/6} )- \frac{\beta_2\eps^\xi}{\zeta -\mathcal{A}_0\eps}  \right)
$$
$$
\beta_{\alpha \bar{\alpha}}^{(\eps,-)}=\frac{ b}{2}\left( 1-\frac{\beta_2}{2} \right) 
\frac{(\proportion_{\bar{\alpha}}-\eps^{1/6} )
(\zeta -\mathcal{A}_0\eps)}{\zeta +\mathcal{A}_0\eps+ \eps^\xi}
$$
$$ 
\delta_{\alpha}^{(\eps,-)}= b+c\left(\mathcal{A}_0\eps+\eps^\xi\right) 
$$
and $\zeta=(b-d)/c$.
Note that for $(\alpha,\alpha') \in \mathfrak{A}^2$ and $* \in \{-,+\}$, the applications 
$\varepsilon\mapsto \beta_{\alpha \alpha'}^{(\eps,*)}$ and $\varepsilon\mapsto\delta_{\alpha}^{(\eps,*)}$ 
are continuous and converge respectively as $\varepsilon\to0$ to $\bar{\beta}_{\alpha,\alpha'}$ and 
$b$ which are the birth and death rates of the process $\bar{\bf N}$ introduced in \eqref{defbarN}. Moreover
$\beta_{\alpha \alpha'}^{(\eps,+)}$ and $\delta_{\alpha}^{(\eps,-)}$ 
(resp. $\beta_{\alpha \alpha'}^{(\eps,-)}$ and $\delta_{\alpha}^{(\eps,+)}$) 
are increasing (resp. decreasing) when $\eps$ increases. 

Let us denote for $* \in \{-,+\}$ and $\alpha \in \mathfrak{A}$ by $q_{\alpha}^{(\eps,*)}$ 
the extinction probability of the process 
$N^{(\eps,*)}$ with initial 
state $\mathbf{e}_{\alpha }$.
As the extinction probability of a supercritical branching process is continuous (see Appendix \ref{app:proba}) with respect to the birth and death rates 
of this process, increases with the death rate and
decreases with the birth rate, we find for $\alpha \in \mathfrak{A}$ that 
\begin{equation} \label{diff_qeps} 
 0 \leq q_{\alpha}^{(\eps,-)}-q_{\alpha}^{(\eps,+)} 
\underset{\eps \to  0}{\to} 0,
\end{equation}
and 
$$ 
q_{\alpha}^{(\eps,+)} \leq q_{\alpha} \leq q_{\alpha}^{(\eps,-)},
$$
where we recall that $q_{\alpha}$ has been defined by \eqref{def_probaq} for the process $\bar {\bf N}$. In other words, for $*\in \{-,+\}$,
\begin{equation}\label{ineq_eta}
\left|q_{\alpha}^{(\eps,*)} - q_{\alpha} \right|=o_\eps(1).
\end{equation}

Since the coupling is only valid on 
$\varSigma_\eps$, we still need to prove that the probabilities of extinction and invasion of the actual process ${\bf N}$ are also given by $q_{\alpha}$ and $1-q_{\alpha}$ 
respectively. To this aim, let us introduce the following stopping times, for $* \in \{-,+\}$,
\begin{equation} \label{defTNproccoupl}
 \forall x \in \R^+,\quad  T^{(\eps,*)}_x:= \inf \{ t > 0,N^{(\eps,*)}(t) = \lfloor K x \rfloor \}.
\end{equation}
Recall that, on $\varSigma_\eps$, the coupling \eqref{couplage1} is satisfied and thus
\begin{equation}\label{ineq_1}
\P\left(T^{(\eps,-)}_{\eps^\xi} < T^{(\eps,-)}_0, \varSigma_\eps\right)
\leq \P\left(T^P_{\eps^\xi} < T^P_0, \varSigma_\eps\right)
\leq \P\left(T^{(\eps,+)}_{\eps^\xi} < T^{(\eps,+)}_0, \varSigma_\eps\right).
\end{equation}
Indeed, if a process reaches the size $\eps^\xi K$ before extinction, it is also the case for a larger process. 
However, $\varSigma_\eps$ is independent from ${\bf N}^{(\eps,-)}$ and ${\bf N}^{(\eps,+)}$, thus  with \eqref{phase_1},
\begin{equation*}
\begin{aligned}
\limsup_{K \to \infty} \ \P\left(T^{(\eps,*)}_{\eps^\xi} < T^{(\eps,*)}_0, \varSigma_\eps\right)&
=\limsup_{K \to \infty} \ \P\left(T^{(\eps,*)}_{\eps^\xi} < T^{(\eps,*)}_0\right)\P\left( \varSigma_\eps\right)\\
&\geq (1-q^{(\eps,*)}_{(\alpha)})(1-C(\A_0,\xi)\eps^{1/12}).
\end{aligned}
\end{equation*}
 Then letting $K$ go to infinity in \eqref{ineq_1}, we find
\begin{equation*}
(1-q^{(\eps,-)}_{(\alpha)})(1-C(\A_0,\xi)\eps^{1/12})
\leq \limsup_{K\to +\infty}\P\left(T^P_{\eps^\xi} < T^P_0, \varSigma_\eps\right)
\leq (1-q^{(\eps,+)}_{(\alpha)})(1-C(\A_0,\xi)\eps^{1/12}).
\end{equation*}
Finally, adding \eqref{phase_1} and \eqref{ineq_eta} we get
\begin{equation}\label{ineq_fin1eretape}
\begin{aligned}
\limsup_{K\to\infty}& \left|\P(T^P_{\eps^\xi } < T^P_0 \wedge R_{\mathcal{A}_0 \eps} 
 \wedge U_{ \eps^{1/6}})-(1-q_{\alpha})\right|\\
 &\leq  \limsup_{K\to\infty} \left|\P(T^P_{\eps^\xi } < T^P_0 , \varSigma_\eps)-(1-q_{\alpha})\right|+ \limsup_{K\to\infty}\left|\P({\varSigma_\eps}^c) \right|
 =o_\eps(1).
\end{aligned}
\end{equation}
Equation~\eqref{eq2prop} is derived similarly.\\

It remains to prove that in the case of invasion (which happens with probability $1-q_{\alpha}$), the time before reaching 
size $K\eps^\xi$ is of order $\log K/\lambda$, where we recall that $\lambda$ is the maximal eigenvalue of the matrix 
$J$ defined in \eqref{matrix}, and that in the case of invasion, $\lambda$ is positive.

We denote by $\lambda^{(\eps,*)}$ the maximal eigenvalue of the mean matrix for the process
$N^{(\eps,*)}$. This eigenvalue is positive for $\eps$ small enough, and converges to $\lambda$ when $\eps$ converges to $0$. 
In other words there exists a nonnegative function $\eta$ going $0$ at $0$ such that, for any $\eps$ small enough,
\begin{equation}\label{limitlambda}
\left|\frac{\lambda^{(\eps,*)}}{\lambda}-1\right|\leq \frac{\eta (\eps)}{2}.
\end{equation}
Thus, let us fix $\eps$ small enough such that the previous inequality holds. 
Then from the coupling \eqref{couplage1}, which is true on $\varSigma_\eps$,
\begin{equation}\label{ineq_dessous}
\P \left( T_{\eps^\xi }^{(\eps,-)}\leq  T_0^{(\eps,-)} \wedge \frac{\ln K}{\lambda}(1+ \eta(\eps)), \varSigma_\eps\right)
\leq \P \left( T^P_{\eps^\xi }\leq T^P_0 \wedge \frac{\ln K}{\lambda}(1+ \eta(\eps)), \varSigma_\eps\right).
\end{equation}
Once again, with independence between $\varSigma_\eps$ and $N^{(\eps,*)}$, using 
classical results on bitype branching processes (see \cite{athreya1972branching}), 
and \eqref{limitlambda}, yields that for $\eps$ small enough (at least such that $\eta(\eps)<1$),
\begin{equation*}
\begin{aligned}
\liminf_{K \to \infty} \ & \P \left( T_{\eps^\xi }^{(\eps,-)}\leq T_0^{(\eps,-)} \wedge \frac{\ln K}{\lambda}(1+ \eta(\eps)), \varSigma_\eps\right)\\
&=\liminf_{K \to \infty} \P \left(  T_{\eps^\xi}^{(\eps,-)} \leq \frac{\ln K}{\lambda}(1+ \eta(\eps))\right) \P\left( \varSigma_\eps\right)\\
&\geq  \liminf_{K \to \infty} \  \P\left( T_{\eps^\xi}^{(\eps,+)} \leq \frac{\ln K}{\lambda^{(\eps,-)}}\left(1-\frac{\eta(\eps)}{2}\right)(1+ \eta(\eps))\right) 
\P\left( \varSigma_\eps\right)\\
&\geq  \liminf_{K \to \infty} \  \P\left( T_{\eps^\xi}^{(\eps,+)} \leq \frac{\ln K}{\lambda^{(\eps,-)}}\left(1 + \frac{\eta(\eps)-\eta^2(\eps)}{2}\right)\right) 
\P\left( \varSigma_\eps\right)\\
& \geq \left(1-q^{(\eps,-)}_{(\alpha)}\right)\left(1-C(\A_0,\xi)\eps^{1/12}\right).
\end{aligned}
\end{equation*}
In addition with \eqref{ineq_dessous}, we deduce that for any small $\kappa>0$
\begin{equation*}
\begin{aligned}
\liminf_{K\to\infty} \P\left(T_{\eps^\xi }^P < T_0^P \wedge R_{\mathcal{A}_0 \eps} 
 \wedge U_{\eps^{1/6}}, T_{\eps^\xi}^P \leq \frac{\ln K}{\lambda}(1+ \eta(\eps)) \right)
 \geq 1-q_{\alpha} - 3\kappa,
\end{aligned}
\end{equation*}
as soon as $\eps$ is small enough. 

We can prove in a similar way, using the upper bound ${\bf Z}^{(\eps,+)}$ of the coupling, that
\[
\limsup_{K\to\infty} \P\left(T_{\eps^\xi }^P < T_0^P \wedge R_{\mathcal{A}_0 \eps} 
 \wedge U_{\eps^{1/6}}, T_{\eps^\xi}^P \leq \frac{\ln K}{\lambda}(1- \eta(\eps)) \right)\leq 1-q_{\alpha}+ 3\kappa.
\]
Putting all pieces together, we conclude the proof of 
Proposition \ref{prop_equiv_inv}.

\subsection{Mean-field phase}\label{subsectionMeanField}

Once the mutant population size has reached an order $K$, the mean-field approximation \eqref{eq:syst}
becomes a good approximation for the population dynamics (cf Lemma~\ref{prop:largepop}). An important question however is the initial condition of the dynamical system used as 
an approximation. Indeed, depending on the initial state, the system \eqref{eq:syst} may converge to various equilibria (see Appendix \ref{app:equilibre} for a study of these 
equilibria). 
The initial state to be considered for the dynamical system and the convergence to a stable equilibrium are the subjects of Sections \ref{section_mutprop} 
and \ref{section_convsystdyn}, respectively.

\subsubsection{Mutant A/a proportions} \label{section_mutprop}

We have seen that when \eqref{eq:noStabNoMig} is satisfied, then the mutant population dynamics is close to that of the supercritical bitype branching 
process $\bar{\mathbf{N}}$ defined in \eqref{defbarN}.
For such a process we are able to control the long time proportion of the different types of individuals. More precisely, Kesten-Stigum theorem 
(see \cite{georgii2003supercritical} for instance) ensures the following property, if $\lambda$ is positive:
$$ \frac{(\bar{N}_A(t),\bar{N}_a(t))}{\bar{N}_A(t)+\bar{N}_a(t)} \underset{t \to \infty}{\to} (\pi_A,\pi_a) \quad \text{almost surely}  $$
on the event of survival of $\bar{\mathbf{N}}$, where $\pi$ is the positive left eigenvalue of $J$ associated to $\lambda$ such that $\pi_A+\pi_a=1$.

The next proposition states that with a probability close to one for large $K$, if the mutant population reaches the size $\eps K$,
we may choose a time when the proportion of type $A$ individuals in the $P$-population belongs to $[\pi_A-\delta,\pi_A+\delta]$, with $\delta>0$ small.

\begin{prop} \label{prop_proportion}
 Let $C>2$ be such that
$$ C \left( \frac{\max\{q_{A},q_{a}\}}{C-1} \right)^{1-1/C} <1.$$
 Assume that $\proportion_A>1/2$ and that \eqref{eq:noStabNoMig} holds.
Let $\delta>0$ such that $\pi_A-\delta >1/2$. Then under the same assumptions as Proposition \ref{prop_equiv_inv}, 
\begin{multline*} \liminf_{K \to \infty} \P \bigg( \exists t\in \Big[  T^P_{\eps}, T^P_{\sqrt{\eps}}\Big], 
\frac{\eps K}{C} \leq N_P(t) \leq \sqrt{\eps} K, \\
\pi_A-\delta< \frac{N_{AP}(t)}{N_P(t)} <\pi_A+\delta \
\Big| \ T^P_{\sqrt{\eps}}<T^P_0 \wedge R_{ \mathcal{A}_0\eps} 
\wedge U_{\eps^{1/6}}  \bigg) \geq 1- o_\eps(1). \end{multline*}
\end{prop}

\begin{proof}
If 
 $$ \pi_A-\delta<\frac{N_{AP}(T^P_{\eps})}{N_P(T^P_{\eps})}<\pi_A+\delta $$
 there is nothing to show. Thus we assume that 
 $$ \frac{N_{AP}(T^P_{\eps})}{N_P(T^P_{\eps})}\leq  \pi_A-\delta. $$
 The symmetric case, when $N_{AP}(T^P_{\eps})/N_P(T^P_{\eps})\geq  \pi_A+\delta$ can be treated with similar arguments. Then, we introduce the event
 $$ \tilde{\varSigma}_\eps:= \{ T^P_{\sqrt{\eps} }< T^P_0 \wedge R_{ \mathcal{A}_0 \eps} 
\wedge U_{\eps^{1/6}} \} $$
on which all calculus will be done.\\
 Our first aim is to prove that the time interval $[T^P_{\eps},T^P_{\sqrt{\eps}}]$ is large when $\eps$ is small and that the mutant population size is not too small on this interval.
 Precisely, we introduce, for any $\eps>0$, the stopping time 
 $$ T^{(\eps)}_{\eps/C}:= \inf \{ t \geq T^P_{\eps}, N_P(t)\leq \eps K/C \}, $$
 where $C$ is the constant introduced in Proposition~\ref{prop_proportion}, and we want to prove that
  the stopping time $T^P_{\sqrt{\eps}}$ is larger than  $T^P_{\eps}+\ln\ln(1/\eps)$ and smaller than 
  $T^{(\eps)}_{\eps/C}$. On the one hand, we obtain from coupling \eqref{couplage1}, satisfied on $\tilde{\varSigma}_\eps$, and Lemma \ref{lemnedescendpas}, that
\begin{equation}\label{majTseps1}
 \lim_{K \to \infty} \P\left(T^{(\eps)}_{\eps/C}<T^P_{\sqrt{\eps}} | \tilde{\varSigma}_\eps\right) =0.
\end{equation}
On the other hand, we obtain from Lemma \ref{yule} that 
\begin{equation}\label{majTseps2}
\lim_{K\to\infty} \P\left(T_{\sqrt{\eps}}^P\leq T^P_\eps+\ln\ln 1/\eps | \tilde{\varSigma}_\eps\right)\leq \sqrt{\eps}\left(\ln 1/\eps\right)^{b(1+\beta_1)},
\end{equation} 
since the process of the total size of $P$-individuals is always stochastically bounded from above by a Yule process with birth rate $b(1+\beta_1)$.
Notice that Lemmas \ref{lemnedescendpas} and \ref{yule} can be applied here because, as we assumed \eqref{eq:noStabNoMig}, 
the mutant $P$ invades with a positive probability and the approximating process $\bar{\mathbf{N}}$ is supercritical.

With this in mind, we are now interested in the dynamics of the fraction of $A$-individuals in the $P$-population. 
Our aim is to find a suitable lower bound to $N_{AP}(t)/N_P(t)$ to prove that this fraction cannot stay below $\pi_A-\delta$ on the interval $[T^P_\eps,T^P_{\sqrt{\eps}}]$ with a 
probability close to 1.
The fraction is a semi-martingale and can be decomposed as
$$ \frac{N_{AP}(t)}{N_P(t)}=\frac{N_{AP}(T^P_{\eps})}{N_{P}(T^P_{\eps})}+M_P(t)+V_P(t), $$
for any $t \geq T^P_\eps$, with $M_P$ a martingale and $V_P$ a finite variation process.

Let us start with the martingale part, $M_P$. Its predictable quadratic variation can be obtained as in \eqref{def_crochet} with $P$ replacing 
$p$ and by integrating between $T^P_\eps$ and $t$ instead of $0$ and $T^P_\eps$. It gives the bound
\begin{align*}
\langle M_P \rangle_t & \leq C_0 (t- T^P_{\eps}) \sup_{T^P_{\eps}\leq s \leq t} \frac{1}{N_{P}(s)-1},
\end{align*}
where $C_0$ is a finite constant. Hence 
\begin{equation}\label{majVbracket}
  \langle M_P \rangle_{(T^P_{\eps}+\ln\ln 1/\eps) \wedge T^{(\eps)}_{\eps/C}}\leq \frac{C_0 \ln \ln 1/\eps}{\eps K/C-1}
\end{equation}
and 
\begin{equation} \label{majM}
\begin{aligned}
&\limsup_{K \to \infty} \P\left( \sup_{T^P_{\eps} \leq t \leq (T^P_{\eps}+\ln\ln 1/\eps) } |M_P(t)|\geq \eps \Big| 
\tilde{\varSigma}_\eps\right) \\
&\hspace{0.5cm}\leq \limsup_{K \to \infty}\left[ \P\left( \sup_{T^P_{\eps} \leq t \leq (T^P_{\eps}+\ln\ln 1/\eps) \wedge 
T^{(\eps)}_{\eps/C}}
 |M_P(t)|\geq \eps \Big| \tilde{\varSigma}_\eps\right)
 + \P\left(T^{(\eps)}_{\eps/C}<T^P_\eps+\ln \ln 1/\eps \Big| \tilde{\varSigma}_\eps\right) \right]\\ 
 &\hspace{0.5cm}\leq
 \limsup_{K \to \infty} \frac{1}{\eps^2} \E\left[\langle M_P \rangle_{T^P_{\eps}+\ln\ln 1/\eps \wedge 
 T^{(\eps)}_{\eps/C}} \Big| \tilde{\varSigma}_\eps\right]+\sqrt{\eps}\left(\ln 1/\eps\right)^{b(1+\beta_1)} = \sqrt{\eps}\left(\ln 1/\eps\right)^{b(1+\beta_1)},
 \end{aligned}
\end{equation}
using Doob's martingale inequality to obtain the third line, and \eqref{majTseps1}, \eqref{majTseps2} and \eqref{majVbracket} for the last one. 
In particular, the martingale is larger than $-\eps$ with a probability close to one.

It remains to deal with the finite variation process $V_P$. It\^o's formula with jumps
gives the following formulation of $V_P$:
\begin{equation*} 
V_P(t)= \int_{T^P_{\eps}}^{t} P^{(s)}\left[\frac{N_{AP}(s)}{N_P(s)}\right]
 \frac{N_P(s)}{N_P(s)+1}ds,
 \end{equation*}
with
$$
  P^{(s)}[X]:= \Big(\beta^{(P)}_{AA} (\mathbf{N}(s)) X+\beta^{(P)}_{aA}(\mathbf{N}(s))(1-X)\Big)(1-X)
 -  \Big(\beta^{(P)}_{aa}(\mathbf{N}(s)) (1-X)+\beta^{(P)}_{Aa}(\mathbf{N}(s)) X\Big)X,
$$
and
 $\beta^{(P)}_{AA}$ and $\beta^{(P)}_{Aa}$ are defined by \eqref{def_betaAA} and \eqref{def_betaAa}.
Notice that, when $\eps$ is small, the polynomial function $P^{(s)}$ is close, on the interval $[0,1]$, to the polynomial function
$$
  P[X]:= \Big(\bar{\beta}_{AA}  X+\bar{\beta}_{aA}(1-X)\Big)(1-X)
 -  \Big(\bar{\beta}_{aa}(1-X)+\bar{\beta}_{Aa} X\Big)X
$$
where the functions $\bar{\beta}_i,i\in\mathcal{G}$ are defined in \eqref{transi_rates}.
Since $P[0]>0$, $P[1]\leq 0$ the equation $\dot{x}= P[x]$ has a unique positive equilibrium in $(0,1]$. Since $(\pi_A,1-\pi_A)$ 
is a left eigenvector of matrix \eqref{matrix}, direct computation ensures that $\pi_A$ is a root of $P$ and thus 
corresponds to this equilibrium. Moreover, since $\rho_{A}>1/2$ we obtain that  $P[1/2]>0$, and we deduce that $\pi_A>1/2$ (therefore, $\delta$ is well defined)
and that there exists a positive $\theta$ such that for any $x<\pi_A-\delta$, $P[x]>\theta$. 
Using the continuity of polynomial functions with respect to their coefficients, we deduce 
the following property conditioning on $\tilde{\varSigma}_\eps$ and for $\eps$ small enough:
\begin{equation}\label{minfrac}
\forall s\in \left[T^P_{\eps}, T^P_{\sqrt{\eps}}\right], \forall \ x \in \left(0,\pi_A-\delta\right), \quad  
P^{(s)}\left[x\right]\geq \frac{\theta}{2}>0. 
\end{equation}
Let us introduce
$$ \tau^{(\eps)}_{A}:= \inf \left\{ t \geq T^P_{\eps}, \frac{N_{AP}(t)}{N_P(t)}\geq \pi_A-\delta \right\}. $$
From \eqref{majM} and \eqref{minfrac}, we thus obtain that, conditioning on $\tilde{\varSigma}_\eps$, 
for any $t\in [T^P_{\eps }, (T^P_\eps+\ln \ln 1/\eps) \wedge \tau^{(\eps)}_{A}]$
\begin{equation} \label{montee_prop}
\pi_A-\delta \geq \frac{N_{AP}(t)}{N_P(t)}
\geq  \frac{\theta}{4}\left(\ln \ln 1/\eps \wedge (\tau^{(\eps)}_{A}-T^P_\eps)\right)-\eps,
\end{equation}
with a probability higher than $1-\sqrt{\eps}\left(\ln 1/\eps\right)^{{b(1+\beta_1)}}$. 
Since $\frac{\theta}{4}\ln \ln 1/\eps -\eps$ converges to $+\infty$ with $\eps$, $\tau^{(\eps)}_{A}$ is smaller 
than $T^P_\eps+\ln \ln 1/\eps$ and so it is smaller than  $T^P_{\sqrt{\eps}}$ with a probability close to one
(conditioning on $\tilde{\varSigma}_\eps$), as soon as $\eps$ is sufficiently small, according to \eqref{majTseps2}. 

Finally, notice that each step of the process $N_{AP}(t)/N_P(t)$ is smaller than $(\eps K/C+1)^{-1}$, hence it is smaller than $\delta$ as 
soon as $K$ is sufficiently large. Thus, after time $\tau^{(\eps)}_A$, the process will belong to the interval $[\pi_A-\delta,\pi_A+\delta]$, 
at least for some times, if $K$ is sufficiently large.
This ends the proof of Proposition~\ref{prop_proportion}.

\end{proof}

\subsubsection{Convergence of the dynamical system} \label{section_convsystdyn}

In this section, we will study the behaviour of the dynamical system \eqref{eq:syst} after the 'stochastic' phase.

The following proposition states that the equilibrium without mutant is unstable under condition \eqref{eq:noStabNoMig},
and Proposition \ref{prop:convergence} states the convergence of the solution to \eqref{eq:syst} under suitable conditions.

\begin{prop}
\label{prop:invasibilite}
Assume that \eqref{eq:noStabNoMig} holds. Then for every $\proportion_A\in[0,1]$,
\begin{itemize}
\item the equilibrium $( 0, \proportion_A (b-d)/c,0,(1-\proportion_A)(b-d)/c)$ is unstable
\item the branching process $\bar{N}$ whose transition rates are given in \eqref{transi_rates} is supercritical
\end{itemize}
On the opposite, if \eqref{eq:noStabNoMig} does not hold, the largest eigenvalue of the Jacobian matrix for \eqref{eq:syst} is $0$. In any case, the equilibrium 
$( 0, \proportion_A (b-d)/c,0,(1-\proportion_A)(b-d)/c)$ is non-hyperbolic.
\end{prop}

\begin{proof}
We compute the Jacobian matrix of system \eqref{eq:syst} at the equilibrium $( 0, \proportion_A (b-d)/c,0,(1-\proportion_A)(b-d)/c)$, 
and obtain when reordering lines and columns $(z_{AP}, z_{aP},z_{Ap},z_{ap})$
$$
\left(\begin{array}{cc}
 J = \frac{b}{2} \left( \begin{matrix}
\proportion_A\beta_1-(1-\proportion_A)\left(\frac{\beta_2}{2}+1 \right) & -\proportion_A\left(\frac{\beta_2}{2}-1\right) \\
- (1-\proportion_A)\left(\frac{\beta_2}{2}-1\right)           & (1-\proportion_A)\beta_1-\proportion_A\left(\frac{\beta_2}{2}+1\right) 
\end{matrix} \right) & *\\
0& \tilde{J}=-(b-d)\left(\begin{matrix}
\proportion_A& \proportion_a\\
\proportion_A& \proportion_a
\end{matrix}\right)
\end{array} \right)
$$
Therefore the eigenvalues of this matrix are the eigenvalues of the two sub-matrices $J$ and $\tilde{J}$.
The eigenvalues of $\tilde{J}$ are $0$ and $-(b-d)<0$. \\
Let us notice that the matrix $J$ admits a positive eigenvalue if and only if either $Tr(J)>0$ or $\Delta(J)<0$ where
\begin{equation*}
\label{eq:TrAndDetLinNoMig}
\left\{ 
\begin{array}{lll}
Tr(J)= \frac{b}{2}\left(\beta_1-\left(\frac{\beta_2}{2}+1 \right)\right) \\
\Delta(J) = \frac{b^2}{4}\left(
	\proportion_A(1-\proportion_A)(\beta_1+\beta_2)(\beta_1+2) - \frac{\beta_1}{2}(\beta_2+2)
\right)
\end{array}
\right. 
\end{equation*} 
We thus obtain that the equilibrium under consideration is unstable if one of the following 
conditions is satisfied:
\begin{equation*}
\beta_1 > \left(\frac{\beta_2}{2}+1 \right)
\quad\text{or}\quad\proportion_A(1-\proportion_A) < \frac{\beta_1 (\beta_2+2)}{2(\beta_1+\beta_2)(\beta_1+2)}.
\end{equation*}
But from a functional study, we can prove that the function $\beta_1\mapsto \beta_1 (\beta_2+2)/(2(\beta_1+\beta_2)(\beta_1+2))$ is larger than $1/4$ for any $\beta_1\in ]\beta_2,\beta_2/2+1]$. 
This concludes the proof for the stability of the equilibrium point $( 0, \proportion_A (b-d)/c,0,(1-\proportion_A)(b-d)/c)$.
Concerning the bitype branching process $\bar{\textbf{N}}$, recall that $J$ is also the mean matrix associated to it. As a consequence, $\bar{N}$ is supercritical if and only if the maximal eigenvalue of $J$ is positive, and the conditions are the same.
\end{proof}

\begin{prop}
\label{prop:convergence}
Let us consider an initial condition $\mathbf{z}^0$ 
such that $z_{Ap}(0)>z_{ap}(0)$ and $z_{AP}(0)>z_{aP}(0)$. Let us furthermore assume that one of the following 
conditions is satisfied:
\begin{equation}
\label{eq:condconv}
\beta_1 > \beta_2
\quad\text{or}\quad\frac{z_A(0)z_a(0)}{z(0)^2} < \frac{\beta_1 (\beta_2+2)}{2(\beta_1+\beta_2)(\beta_1+2)}.
\end{equation}
Then the solution $\mathbf{z}^{(\mathbf{z}^0)}$ of the system \eqref{eq:syst} converges as $t\to\infty$ toward 
$$(((1+\beta_1) b-d)/c,0,0,0).$$
\end{prop}

\begin{proof}
To prove the convergence we will consider the diversity at locus $A/a$ using the quantity 
$$D:=\frac{z_Az_a}{z^2}$$ 
and prove that this quantity converges to $0$. The differential equation followed by $D$ is:
\begin{equation} \label{dotDA}
\begin{aligned}
\dot{D}&= \frac{b}{z^2}\left(
	(z_{AP}z_{a}+z_{A}z_{aP})\left(
		D(\beta_1+2\beta_2)
		- \frac{\beta_2}{2}
	\right)
	-D \beta_1 (z_{AP}z_A+n_{aP}z_a)
\right)\\
&= \frac{b}{z^2}\left(
	(z_{AP}z_{a}+z_{A}z_{aP})\frac{\beta_2}{2}\left(
		4D-1\right)
	-D \beta_1 (z_{AP}-z_{aP})^2-D  {\beta_1} (z_{AP}-z_{aP})(z_{Ap}-z_{ap})\right)\\
&\leq   -D  \frac{b\beta_1}{z^2} (z_{AP}-z_{aP})(z_{Ap}-z_{ap}),
\end{aligned}
\end{equation}
since $D$ is always less than $1/4$.
Let us introduce the function
$$\Pi(t):=(z_{AP}(t)-z_{aP}(t))(z_{Ap}(t)-z_{ap}(t)).$$
Under the assumption of Proposition \ref{prop:convergence}, $\Pi(0)>0$. We want to prove that 
$\Pi(t)>0$ for all $t>0$.
We start by computing the derivative of this quantity :
\begin{equation}
 \begin{aligned}
  \frac{d \Pi}{dt} =& 2(b-d-cz)\Pi\\
  &+\frac{b(z_{Ap}-z_{ap})}{z}\left[ \beta_1(z_{AP}(z_{AP}+\frac{z_{Ap}}{2})-z_{aP}(z_{aP}+\frac{z_{ap}}{2})) +(z_{aP}z_{Ap}-z_{AP}z_{ap}) \right]\\
  &+\frac{b(z_{AP}-z_{aP})}{z} \left[\beta_1(z_{Ap}\frac{z_{AP}}{2}-z_{ap}\frac{z_{aP}}{2})-(z_{aP}z_{Ap}-z_{AP}z_{ap}) \right].
 \end{aligned}
\end{equation}
By reorganizing the terms, we find
\begin{equation}\label{eq_pi2}
 \begin{aligned}
  \frac{d \Pi}{dt}=& \left(2b-2d+b\beta_1\frac{3z_{AP}+2z_{aP}+z_{ap}}{2z}-b\frac{z_{aP}+z_{ap}}{z}-2cz\right)\Pi\\
  &+\frac{b\beta_1}{2z}[z_{AP}(z_{Ap}-z_{ap})^2+z_{ap}(z_{AP}-z_{aP})^2]+\frac{b}{z}[z_{aP}(z_{Ap}-z_{ap})^2+z_{ap}(z_{AP}-z_{aP})^2]\\
  \geq & \left(b-2d-2cz\right)\Pi.
 \end{aligned}
\end{equation}

We thus need information on the dynamics of $z$ to conclude. From \eqref{eq:syst}, we obtain
\begin{align*}
\frac{d z(t)}{dt}& = z(b-d-cz)+ b\frac{\beta_1}{z}(z_{AP}z_A+z_{aP}z_a)- b \frac{\beta_2}{z}(z_{AP}z_a+z_{aP}z_A) \\
&\leq z(t)\left(b(1+\beta_1)-d-cz(t)\right).
\end{align*}
In other words, for any $t\geq 0$, 
$$z(t)\leq z(0) \vee \frac{b(1+\beta_1)-d}{c}.$$
Combining with \eqref{eq_pi2} we deduce that as long as $\Pi(t)>0$,
$$
\frac{d\Pi(t)}{dt}\ge \left(b-2d-2c \left( z(0) \vee \frac{b(1+\beta_1)-d}{c} \right) \right)\Pi(t),
$$ 
and thus 
$$\Pi(t)\geq \Pi(0) e^{-Ct}>0, \qquad \forall t \geq 0. $$

Combining this result with \eqref{dotDA}, we deduce that $D$ is a positive and decreasing quantity and converges to a 
nonnegative value where its derivative $\dot D$ vanishes. We deduce from the fact that all three terms of the second line of \eqref{dotDA} are negative that $\lim_{t\to\infty}D(t) (z_{AP}(t)-z_{aP}(t))^2=0$.

From Proposition \ref{prop:equilibre}, the possible limits are the points 
$$(0,0,0,0), \qquad  \chi_{AP}:=\left(\frac{(1+\beta_1)b-d}{c},0,0,0\right), $$
$$\left(\frac{b(1+(\beta_1-\beta_2)/2)-d}{2c},0,\frac{b(1+(\beta_1-\beta_2)/2)-d}{2c},0\right),$$
and the line 
$$\left(0,\pi\frac{b-d}{c},0,(1-\pi)\frac{b-d}{c}\right) \text{ with } \pi\in[1/2,1].
$$
 The proof of Proposition~\ref{prop:equilibre} \ (i) ensures that no positive trajectory converges to the null point. 
 Moreover, we proved that $D$ is decreasing. As it starts from $D(0)=z_{A}(0)(z(0)-z_{A}(0))/z(0)^2<1/4$, 
 the set of possible limits is thus restricted to the point $\chi_{AP}$ or the line
\begin{equation} \label{eq_possiblelimit}
\left(0,\pi\frac{b-d}{c},0,(1-\pi)\frac{b-d}{c}\right), \quad \pi \in [z_{A}(0)/z(0),1].
\end{equation}
As $D$ is decreasing, the trajectory cannot oscillate close to the line of~\eqref{eq_possiblelimit}. Hence, if it approaches the line in large time, it should converge to a point 
of this line.
Assume that it converges to $\left(0,\pi(b-d)/c,0,(1-\pi)(b-d)/c\right)$. 
Note that, from Assumption~\eqref{eq:condconv},
$$
\beta_1 > \beta_2 \quad\text{or}\quad
\pi(1-\pi)  \leq D(0)  < \frac{\beta_1 (\beta_2+2)}{2(\beta_1+\beta_2)(\beta_1+2)}
$$
meaning that the equilibrium $\left(0,\pi(b-d)/c,0,(1-\pi)(b-d)/c\right)$ is unstable.
Hence, from Perron-Frobenius Theorem and Proposition~\ref{prop:invasibilite}, there exists $(\gamma_1 , \gamma_2)$ 
left positive principal eigenvector of the matrix $J$ (see the proof of Proposition~\ref{prop:invasibilite}) which is positive and associated to a positive eigenvalue $\lambda$. 
Using similar computations, we obtain that in the neighbourhood of $\left(0,\pi(b-d)/c,0,(1-\pi)(b-d)/c\right)$
$$
\gamma_1 \dot z_{AP} +\gamma_2 \dot z_{aP}\geq \frac{\lambda}{2} (\gamma_1 z_{AP}+\gamma_2 z_{aP}).
$$
Thus, as soon as $z_{AP}$ and $z_{aP}$ are not equal to 0, the quantity $\gamma_1 z_{AP}+\gamma_2 z_{aP}$ will grow exponentially fast when the trajectory is close to 
$\left(0,\pi(b-d)/c,0,(1-\pi)(b-d)/c\right)$, and therefore
it cannot converge to this state.
\end{proof}

\subsection{Extinction} \label{section_ext}

After the deterministic phase, the process  is close to the state $( (b(\beta_1+1)-d)/c,0,0,0 )$. 
In this subsection, we are interested in estimating the time before 
the extinction of all but $AP$-individuals in the population.
We also need to check that the $AP$-population size stays close to its equilibrium during this extinction time.
We recall here the definition of the set $S_\eps$ and the stopping time $ T_{S_\eps}$ in \eqref{defSeps} and \eqref{defTKSeps}, respectively:

\begin{equation*}S_\eps := \left[ \frac{b(1+\beta_1)-d}{c}-\eps , \frac{b(1+\beta_1)-d}{c}+\eps \right] \times \{0\} \times \{0\}\times \{0\}, \end{equation*}
\begin{equation*} T_{S_\eps}:= \inf \{ t \geq 0, \mathbf{Z}^K(t) \in S_{\eps} \}. \end{equation*}

\begin{prop}
 \label{prop_ext}
There exist two positive constants $\varepsilon_0$ and $C_0$ such that for any $\varepsilon\leq \varepsilon_0$,
 if there exists $\eta\in ]0,1/2[$ that satisfies 
 $$\left|Z_{AP}(0)-\frac{b(1+\beta_1)-d}{c}\right| \leq \eps \quad \text{and} \quad
\eta\eps/2 \leq Z_{Ap}(0)+Z_{ap}(0)+Z_{aP}(0) \leq \eps/2,$$
then 
 $$
\begin{aligned}
 &\text{for all } C>2/(b\beta_1)+C_0\eps, &\P\left(T_{S_\eps}\leq C \log(K)\right) \underset{K\to +\infty}{\to} 1,\\
 &\text{for all } 0\leq C <2/(b\beta_1)-C_0\eps, & \P\left(T_{S_\eps}\leq C \log(K)\right) \underset{K\to +\infty}{\to} 0.
\end{aligned}
$$
\end{prop}

\begin{proof}
 This proof is very similar to the proof of Proposition 4.1 in \cite{Coron2018}. 
 We thus only detail parts of the proof that are significantly different.
 
 Following these ideas, we prove that as long as the sum $Z_{Ap}+Z_{ap}+Z_{aP}$ is small (lower than $\eps$), 
 the process $Z_{AP}$ stays close to $(b(1+\beta_1)-d)/c$.

 Then, we can bound the death and birth rates of $Z_{Ap}$, $Z_{ap}$ and $Z_{aP}$ under the previous approximation and compare the dynamics of these three processes with the ones of 
$$\left(\frac{\mathcal{N}_{Ap}(t)}{K},\frac{\mathcal{N}_{ap}(t)}{K},\frac{\mathcal{N}_{aP}(t)}{K}, t\geq 0 \right),$$
 where $(\mathcal{N}_{Ap},\mathcal{N}_{ap},\mathcal{N}_{aP} )\in \N^3$ is a three-types branching process with types $Ap$, $ap$ and $aP$ and such that 
\begin{itemize}
 \item any $Ap$-individual gives birth to a $Ap$-individual at rate $b(2+\beta_1)/2$,
 \item any $aP$-individual gives birth to a $aP$-individual at rate $b(1-\beta_2)$,
 \item any individual gives birth to a $ap$-individual at rate $b(2-\beta_2)/4$,
 \item any individual dies at rate $b(1+\beta_1)$.
\end{itemize}
The goal is thus to estimate the extinction time of such a sub-critical three type branching process. 
According to \cite{athreya1972branching} p. 202 and Theorem 3.1 in~\cite{heinzmann2009extinction},
\begin{equation}
\label{eq_probabitype}
\P\Big((\mathcal{N}_{Ap}(t),\mathcal{N}_{ap}(t),\mathcal{N}_{aP}(t))=(0,0,0)\Big)= (1-c_1 e^{rt})^{Z_{Ap}(0) K}(1-c_2 e^{rt})^{Z_{ap}(0) K}(1-c_3 e^{rt})^{Z_{aP}(0) K},
\end{equation}
where $c_1,c_2$ and $c_3$ are three positive constants and $r$ is the largest eigenvalue of
\begin{equation*}
\frac{b}{2}\begin{pmatrix}
-\beta_1 & 1-\frac{\beta_2}{2} & 0\\
0 & -2\beta_1-1-\frac{\beta_2}{2} & 0\\
0 & 1-\frac{\beta_2}{2} & -2\beta_1-2\beta_2
\end{pmatrix},
\end{equation*}
which is $r=-b\beta_1/2$. From~\eqref{eq_probabitype}, we deduce that the extinction time is of order 
$(2/b\beta_1)\log K$ when $K$ tends to $+\infty$ by arguing as in step 2 in the proof of Proposition 4.1 in
\cite{Coron2018}. This concludes the proof.
\end{proof}

\subsection{Proof of Theorem \ref{theo_main}}

The proof strongly relies on the coupling \eqref{couplage1}. More precisely, we consider a trajectory of $\bar{\Nbf}$ (defined in \eqref{defbarN}) coupled with the mutant process. The random variable $B$ of Theorem \ref{theo_main} is then defined as

$$ B:= \mathbf{1}_{\{\bar{T}_0=\infty\}}, $$
which equals $1$ if the process $\bar{\Nbf}$ survives and $0$ otherwise. In particular, $B$ is indeed a Bernoulli random variable with parameter $1-q_\alpha$ where $\alpha\in\mathfrak{A}$ is the genotype of the first mutant individual.\\
Let the function $\eta$ be defined as in Proposition~\ref{prop_equiv_inv}. The convergence in probability claimed in \eqref{res_main} is equivalent to 
\begin{equation}\label{eq_convproba}
 \liminf_{K \to \infty} \P \left( \left\| \left(\frac{ T_{S_\mu} \wedge T^P_0}{\ln K}, \mathbf{1}_{\{ T_{S_\mu} < T^P_0 \}}\right) 
-\left(\frac{1}{\lambda}+\frac{2}{b\beta_1},1\right)B \right\|_1\leq \eta(\eps) \right) \geq 1+o_\eps(1).
\end{equation}
As $\eps$ is as small as we need, we can assume without loss of generality that $\eta(\eps)<1$. In the sequel, we divide the probability into two terms according to the values of $B$ using that $\{B=1\}=\{\bar{T}_0=+\infty\}$, we obtain
\begin{equation}\label{defFG}
\begin{aligned}
\P &\left( \left\| \left(\frac{ T_{S_\mu} \wedge T^P_0}{\ln K}, \mathbf{1}_{\{ T_{S_\mu} < T^P_0 \}}\right) 
-\left(\frac{1}{\lambda}+\frac{2}{b\beta_1},1\right)B \right\|_1\leq \eta(\eps) \right)\\
&=\P \left( \left| \frac{ T_{S_\mu} }{\ln K} 
-\left(\frac{1}{\lambda}+\frac{2}{b\beta_1}\right) \right|\leq \eta(\eps),  T_{S_\mu} < T^P_0, \bar T_0=+\infty \right)\\
&+\P \left( \left| \frac{  T^P_0}{\ln K} 
\right|\leq \eta(\eps), T^P_0< T_{S_\mu}, \bar T_0<+\infty  \right)\\
&=: \mathcal{F}(K,\eps)+ \mathcal{G}(K,\eps).
\end{aligned}
\end{equation}
Let us first consider $\mathcal{G}(K,\eps)$, which is simpler to deal with and which represents the case of extinction of $P$-individuals. We introduce $\mathcal{A}_0$,  $C>2$, $\delta>0$ and $\mathbf{z}(0)$ as 
in Propositions \ref{prop_equiv_inv}, \ref{prop_proportion} and \ref{prop:convergence}.
First of all, notice that
$$
\mathcal{G}(K,\eps) \geq \P \left( \left| \frac{  T^P_0}{\ln K} 
\right|\leq \eta(\eps), T^P_0< T_{S_\mu}, \bar T_0<+\infty, T^P_0<T^P_\eps \wedge R_{\mathcal{A}_0 \eps} 
 \wedge U_{\eps^{1/6}}   \right).
$$
Then if $\eps$ is chosen small enough such that $\eps<\left((b(1+\beta_1)-d)/c-\mu\right)\wedge z_p(0)/\mathcal{A}_0$ and considering our initial conditions, we have
$$T^P_\eps \wedge R_{\mathcal{A}_0 \eps} 
 \wedge U_{\eps^{1/6}}<T_{S_\mu} \text{ a.s.,}
 $$
  hence
 $$
 \left\{T^P_0< T_{S_\mu}, T^P_0<T^P_\eps \wedge R_{\mathcal{A}_0 \eps} 
 \wedge U_{\eps^{1/6}}\right\}=\left\{ T^P_0<T^P_\eps \wedge R_{\mathcal{A}_0 \eps} 
 \wedge U_{\eps^{1/6}}\right\} \text{ a.s.}
 $$
 and
\begin{equation}\label{eq_stepG1}
\mathcal{G}(K,\eps) \geq \P \left( \left| \frac{  T^P_0}{\ln K} 
\right|\leq \eta(\eps), \bar T_0<+\infty, T^P_0<T^P_\eps \wedge R_{\mathcal{A}_0 \eps} 
 \wedge U_{\eps^{1/6}}   \right).
\end{equation}
Moreover
from \eqref{phase_1} to  \eqref{ineq_eta}, and reasoning as in \eqref{ineq_eta} to\eqref{ineq_fin1eretape}, we obtain
\begin{equation}\label{eq_stepG2}
 \limsup_{K \to\infty}  \P\left( \left\{\bar{T}_0 <\infty \right\}\triangle \left\{T^P_0 < T^P_{\eps} \wedge R_{\mathcal{A}_0 \eps} 
 \wedge U_{\eps^{1/6}} \right\} \right) = o_\eps(1)
\end{equation}
and
$$ \limsup_{K \to\infty}  \P\left( \left\{\bar{T}_0 <\infty \right\}\triangle \left\{T^{(\eps,+)}_0 < T^{(\eps,+)}_{\infty} \right\} \right) = o_\eps(1).
 $$
In addition with \eqref{eq_stepG1}, we thus deduce
\begin{equation}\label{eq_stepfinG}
\begin{aligned} 
\liminf_{K \to\infty}\mathcal{G}(K,\eps)\geq &\liminf_{K \to\infty} \P \left( \left| \frac{ \lambda T^P_0}{\ln K}\right|\leq \eta(\eps),T^{(\eps,+)}_0 < T^{(\eps,+)}_{\infty} \right) + 
o_\eps(1)\\
\geq & \liminf_{K \to\infty} \P \left( \left| \frac{ \lambda T^{(\eps,+)}_0}{\ln K}\right|\leq \eta(\eps),T^{(\eps,+)}_0 < T^{(\eps,+)}_{\infty} \right) + o_\eps(1)\\
\geq & \liminf_{K \to\infty} \P \left(T^{(\eps,+)}_0 < T^{(\eps,+)}_{\infty} \right)+ o_\eps(1) = q_\alpha + o_\eps(1),
\end{aligned}
\end{equation}
where the second inequality comes from coupling \eqref{couplage1}. This allows the case of extinction to be processed.\\
{Let us now deal with $\mathcal{F}(K,\eps)$, which represents the case of survival and invasion of $P$-individuals. Firstly, reasoning as for \eqref{eq_stepG2} but with $\xi=1/2$, we can get
$$
 \limsup_{K \to\infty}  \P\left( \left\{\bar{T}_0 =\infty \right\}\triangle \left\{T^P_{\sqrt{\eps}} < T^P_0 \wedge R_{\mathcal{A}_0 \eps} 
 \wedge U_{\eps^{1/6}} \right\} \right) = o_\eps(1).
$$ 
Hence
\begin{multline}
\label{eq_stepF1}\liminf_{K\to \infty}\mathcal{F}(K,\eps)=\\ 
\liminf_{K\to\infty} \P \left( \left| \frac{ T_{S_\mu} }{\ln K} 
-\left(\frac{1}{\lambda}+\frac{2}{b\beta_1}\right) \right|\leq \eta(\eps),  T_{S_\mu} < T^P_0, T_{\sqrt{\eps}}^P < T^P_0 \wedge R_{\mathcal{A}_0 \eps} 
 \wedge U_{\eps^{1/6}} \right)+o_\eps(1).
\end{multline}
We introduce two sets for any $\eps>0$, $\mu>0$,
$$
\mathcal{K}_{\eps}^1=\left[ \pi_A-\delta, \pi_A+\delta \right] \times  \left[\frac{\eps}{C} , \sqrt{\eps}\right] \times
 \left[ \rho_A-\eps^{1/6}, \rho_A+\eps^{1/6}  \right] \times
 \left[ \frac{b-d}{c}-\mathcal{A}_0\eps ,\frac{b-d}{c}+\mathcal{A}_0\eps \right],
$$
$$
\mathcal{K}_{\mu}^2=\left[\frac{b(1+\beta_1)-d}{c}-\frac{\mu}{2},\frac{b(1+\beta_1)-d}{c}+\frac{\mu}{2}\right]\times \left[0,\frac{\mu}{2}\right]^3.
$$
as well as the stopping times
$$
T_{\eps}^1=\inf\left\{t\geq 0, \left( \frac{N_{AP}(t)}{N_P(t)},\frac{N_P(t)}{K},\frac{N_{Ap}(t)}{N_p(t)},\frac{N_p(t)}{K} \right)\in \mathcal{K}_{\eps}^1\right\},
$$
$$
T_{\mu}^2=\inf\left\{t\geq T_{\eps}^1, Z^K(t)\in \mathcal{K}_{\mu}^2\right\}.
$$
Our aim is essentially to prove that the only path to $S_\mu$ is through $\mathcal{K}_{\eps}^1$ and $\mathcal{K}_{\mu}^2$, as presented in the introduction of the paper. Then, using the Markov property and the previous propositions, we want to estimate $T_{S_\mu}$, by dividing $[0,T_{S_\mu}]$ into three parts: $[0,T_{\eps}^1]$, $[T_{\eps}^1,T_{\mu}^2]$ and $[T_{\mu}^2,T_{S_\mu}]$.
From \eqref{eq_stepF1},
\begin{align*}
&\liminf_{K\to \infty}\mathcal{F}(K,\eps)\\
&\geq \liminf_{K\to\infty} \P \bigg( \left| \frac{ T_{S_\mu} }{\ln K} 
-\left(\frac{1}{\lambda}+\frac{2}{b\beta_1}\right) \right|\leq \eta(\eps),\\
& \hspace{4.5cm}  T_{S_\mu} < T^P_0, T^P_{\sqrt{\eps}} < T^P_0 \wedge R_{\mathcal{A}_0 \eps} 
 \wedge U_{\eps^{1/6}}, T_{\mu}^2<T_{S_\mu} \bigg)+o_\eps(1)\\
 &\geq \liminf_{K\to\infty} \P \bigg( \left| \frac{ T_{\eps}^1 }{\ln K} 
-\frac{1}{\lambda} \right|\leq \frac{\eta(\eps)}{3}, 
\left| \frac{ T_{\mu}^2 -T^1_\varepsilon}{\ln K} 
\right|\leq \frac{\eta(\eps)}{3}, 
\left| \frac{ T_{S_\mu} -T_{\mu}^2}{\ln K} 
-\frac{2}{b\beta_1} \right|\leq \frac{\eta(\eps)}{3},\\
& \hspace{4.5cm}  T_{S_\mu} < T^P_0, T^P_{\sqrt{\eps}} < T^P_0 \wedge R_{\mathcal{A}_0 \eps} 
 \wedge U_{\eps^{1/6}}, T_{\mu}^2<T^P_0\wedge T_{S_\mu} \bigg)+o_\eps(1)
\end{align*}}

Then, since for $\eps$ sufficiently small, $R_{\mathcal{A}_0 \eps} 
 \wedge U_{\eps^{1/6}}\leq T_{S_\mu}$ a.s. and using the Markov property at times $T_{\mu}^2$ and $T_{\eps}^1$ we obtain
 \begin{equation}\label{eq_MarkovF}
\begin{aligned}
\liminf_{K\to \infty}\mathcal{F}&(K,\eps)\\
\geq &\liminf_{K\to\infty} \bigg[ \P \bigg( \left| \frac{ T_{\eps}^1 }{\ln K} 
-\frac{1}{\lambda} \right|\leq \frac{\eta(\eps)}{3}, T_{\eps}^1<T^P_0, T^P_{\sqrt{\eps}} < T^P_0 \wedge R_{\mathcal{A}_0 \eps} 
 \wedge U_{\eps^{1/6}}\bigg)\\ 
&\times \inf_{{\bf z}(0)\in \mathcal{K}_{\eps}^1}\P\bigg( \left| \frac{ T_{\mu}^2 }{\ln K} 
\right|\leq \frac{\eta(\eps)}{3}, T_{\mu}^2<T^P_0 \bigg|{\bf Z}(0)={\bf z}(0)\bigg)\\
&\times \inf_{{\bf z}(0)\in \mathcal{K}_{\mu}^2} \P\bigg(
\left| \frac{ T_{S_\mu}}{\ln K} 
-\frac{2}{b\beta_1} \right|\leq \frac{\eta(\eps)}{3}, T_{S_\mu} < T^
P_0 \bigg| {\bf Z}(0)={\bf z}(0) \bigg) \bigg]+o_\eps(1).
\end{aligned}
\end{equation}
To complete the proof it remains to show that r.h.s of \eqref{eq_MarkovF} is close to $1-q_{\alpha}$ when $K$ goes to $\infty$ and $\eps$ is small.
Let us start with the first term. Our aim is to prove that
 \begin{equation}\label{eq_inter1F}
  \liminf_{K\to\infty} \P \bigg( \left| \frac{ T_{\eps}^1 }{\ln K} 
-\frac{1}{\lambda} \right|\leq \frac{\eta(\eps)}{3}, T_{\eps}^1<T^P_0, T^P_{\sqrt{\eps}} < T^P_0 \wedge R_{\mathcal{A}_0 \eps} 
 \wedge U_{\eps^{1/6}}\bigg)= 
 1- q_\alpha+o_\eps(1).
\end{equation}
To this aim, let us notice that the following series of inequalities holds:
\begin{align*}
  & \P \left( \left| \frac{ \lambda T_{\eps}^1 }{\ln K} -1 \right|\leq \eta(\eps),  T^P_{\sqrt{\eps}} < T^P_0 \wedge R_{\mathcal{A}_0 \eps} 
 \wedge U_{\eps^{1/6}} \right)\\
 &\geq 
   \P \left( \left| \frac{ \lambda T_{\eps}^1 }{\ln K} - \frac{\lambda T^P_{\sqrt{\eps} }}{\ln K}\right|\leq \frac{\eta(\eps)}{2},  
    \left|  \frac{ \lambda T^P_{\sqrt{\eps} }}{\ln K} -1 \right|\leq  \frac{\eta(\eps)}{2},  T^P_{\sqrt{\eps}} < T^P_0 \wedge R_{\mathcal{A}_0 \eps} 
 \wedge U_{\eps^{1/6}} \right)\\ 
 &\geq 
    \P \left( \frac{\lambda T^P_{\sqrt{\eps} }}{\ln K}-\frac{\lambda T^P_\eps }{\ln K} \leq \frac{\eta(\eps)}{2},  
   T^P_\eps  \leq T_{\eps}^1 \leq T^P_{\sqrt{\eps}},
    \left|  \frac{\lambda T^P_{\sqrt{\eps} }}{\ln K} -1 \right|\leq  \frac{\eta(\eps)}{2},  T^P_{\sqrt{\eps}} < T^P_0 \wedge R_{\mathcal{A}_0 \eps} 
 \wedge U_{\eps^{1/6}} \right)
\end{align*}
Now, if $A,B,C$ and $D$ are events, we have 
\begin{align*} 
\P(A \cap B \cap C \cap D) &=\P( C \cap D)-  \P\left( (A\cap B)^C \cap C \cap D)  \right)\\
& \geq \P( C \cap D)-  \P\left( A^C \cap D \right)
-  \P\left(B^C  \cap D \right).
\end{align*}
Applying this to the previous series of inequalities yields
\begin{multline}
 \label{series_inegF} \P \left( \left| \frac{\lambda T_{\eps}^1 }{\ln K} -1 \right|\leq \eta(\eps),  T^P_{\sqrt{\eps}} < T^P_0 \wedge R_{\mathcal{A}_0 \eps} 
 \wedge U_{\eps^{1/6}} \right) \geq \\
  \P \left( \left|  \frac{\lambda T_{\sqrt{\eps} }}{\ln K} -1 \right|\leq  \frac{\eta(\eps)}{2},  T^P_{\sqrt{\eps}} < T^P_0 \wedge R_{\mathcal{A}_0 \eps} 
 \wedge U_{\eps^{1/6}} \right) \\
 -  \P \left( \frac{\lambda T^P_{\sqrt{\eps} }}{\ln K}-\frac{\lambda T^P_\eps }{\ln K} \geq \frac{\eta(\eps)}{2},  T^P_{\sqrt{\eps}} < T^P_0 \wedge R_{\mathcal{A}_0 \eps} 
 \wedge U_{\eps^{1/6}} \right)  \\
 -  \P \left( T_{\eps}^1 \notin \left[T^P_\eps,T^P_{\sqrt{\eps}}\right],
     T^P_{\sqrt{\eps}} < T^P_0 \wedge R_{\mathcal{A}_0 \eps} 
 \wedge U_{\eps^{1/6}} \right) 
\end{multline}
Proposition \ref{prop_equiv_inv} implies that the first term in the right hand side of \eqref{series_inegF} satisfies
$$ \liminf_{K \to \infty}   \P \left( \left|  \frac{ T^P_{\sqrt{\eps} }}{\ln K} -\frac{1}{\lambda} \right|\leq  \frac{\eta(\eps)}{2},  T^P_{\sqrt{\eps}} < T^P_0 \wedge R_{\mathcal{A}_0 \eps} 
 \wedge U_{\eps^{1/6}} \right) \geq 1- q_\alpha -o_\eps(1). $$
From Lemma \ref{yule}, we deduce that the second term of the right hand side of \eqref{series_inegF} satisfies  $$\liminf_{K\to\infty}\P \left( \frac{\lambda T^P_{\sqrt{\eps} }}{\ln K}-\frac{\lambda T^P_\eps }{\ln K} \geq \frac{\eta(\eps)}{2},  T^P_{\sqrt{\eps}} < T^P_0 \wedge R_{\mathcal{A}_0 \eps} 
 \wedge U_{\eps^{1/6}} \right)=o_\eps(1).$$ Finally Proposition \ref{prop_proportion} implies that the last term of the right hand side of \eqref{series_inegF} satisfies: $$\liminf_{K\to\infty}\P \left( T^1_\eps \notin \left[T^P_\eps,T^P_{\sqrt{\eps}}\right],
     T^P_{\sqrt{\eps}} < T^P_0 \wedge R_{\mathcal{A}_0 \eps} 
 \wedge U_{\eps^{1/6}} \right)=o_\eps(1).$$
This leads to \eqref{eq_inter1F}.\\
Then we deal with the second term of \eqref{eq_MarkovF} by using Proposition \ref{prop:convergence}. 
Using the continuity of flows of the dynamical system \eqref{eq:syst} with respect to 
the initial condition and the convergence given by Proposition \ref{prop:convergence}, 
we get that there exist $ \eps_0,\delta_0>0 $ such that for all $\eps \leq \eps_0$, $\delta \leq \delta_0$, there exists a $t_{\mu,\delta,\varepsilon}>0$ 
such that for all $t\ge  t_{\mu,\delta,\varepsilon}$
$$\left\|  \mathbf{z}^{(\mathbf{z}^0)}(t)-\left(\frac{b(1+\beta_1)-d}{c},0,0,0\right)\right\|_{\infty}  \le \frac{\mu}{4},$$
for every initial condition $\mathbf{z}^0=(z^0_{AP},z^0_{aP},z^0_{Ap},z^0_{ap}) $ such that $\left( {z_{AP}}/{z_P},z_P,{z_{Ap}}/{z_p},z_p \right)$ belongs to $\mathcal{K}^1_\eps$.\\
Now using Lemma \ref{prop:largepop}, we get that for any $\mu>0$, and $\eps<\eps_0$,
\begin{equation*}
\lim_{K \to \infty} \P \left( T^2_\mu-T^1_\varepsilon \leq t_{\nu,\delta,\varepsilon} | {\bf Z}(0)\in \mathcal{K}^1_\eps  \right)=1.
\end{equation*}
In other words, the second term of \eqref{eq_MarkovF} is close to $1$ when $K$ converges to $\infty$ and $\eps$ is small.\\

Finally, we deal with the third term of \eqref{eq_MarkovF}.
 Applying Proposition \ref{prop_ext} we obtain that there exists $\mu_0$ (defined by $\eps_0$ in Proposition \ref{prop_ext}) such that for $\mu<\mu_0$ and $\eps$ small enough,
\begin{equation} \label{etape3F}  \lim_{K \to \infty} \P \left(  \left| \frac{T_{S_\mu}}{\ln K}-  \frac{2}{b\beta_1} \right| 
\leq \frac{\eta(\eps)}{3}\Bigg| {\bf Z}(0)\in \mathcal{K}^2_\mu  \right)=1. 
\end{equation}
By combining \eqref{eq_inter1F}, the convergence of the second term of \eqref{eq_MarkovF} to $1$ and \eqref{etape3F}, we get
\begin{equation*}
\liminf_{K\to \infty}\mathcal{F}(K,\eps)
\geq 1- q_\alpha+o_\eps(1).
\end{equation*} 
In addition with \eqref{eq_stepfinG} and \eqref{defFG}, we deduce \eqref{eq_convproba}.
Finally \eqref{eq_cv_extin} derives from \eqref{eq2prop} which ends the proof of Theorem \ref{theo_main}.

\appendix

\section{Technical results} \label{sec:app:dynsys}

\subsection{Equilibria of dynamical system~\eqref{eq:syst}}
\label{app:equilibre}

In this section, we study the existence and the stability of some equilibria of the 
dynamical system~\eqref{eq:syst}. For the sake of readability, we explicitly rewrite the dynamical system below
\begin{equation}
\label{eq:systbis}
\left\{ 
\begin{array}{lll}
\dot{z}_{AP} &=& b z_{AP}+ \frac{b}{z}\left[\beta_1 z_{AP}\left( z_{AP}+\frac{z_{Ap}}{2} \right)-\beta_2\left( z_{AP} 
\left( z_{aP}+\frac{z_{ap}}{4}\right)+ z_{Ap}
 \frac{z_{aP}}{4} \right) \right]\\
&& + \frac{b}{2z} \left(z_{aP}z_{Ap}-z_{AP}z_{ap}\right)- (d+c z)z_{AP}\\
\dot{z}_{Ap}&=& b z_{Ap}+ \frac{b}{z}\left[ \beta_1 z_{Ap}\frac{z_{AP}}{2}-\beta_2\left( z_{Ap}  \frac{z_{aP}}{4}+ z_{AP}
 \frac{z_{ap}}{4} \right) \right]\\
&& - \frac{b}{2z} \left(z_{aP}z_{Ap}-z_{AP}z_{ap}\right)- (d+c z)z_{Ap}\\
\dot{z}_{aP} &=& b z_{aP}+ \frac{b}{z}\left[ \beta_1 z_{aP}\left( z_{aP}+\frac{z_{ap}}{2} \right)-\beta_2\left( z_{aP} 
\left( z_{AP}+\frac{z_{Ap}}{4}\right)+ z_{ap}
 \frac{z_{AP}}{4} \right) \right]\\
&& + \frac{b}{2z} \left(z_{AP}z_{ap}-z_{aP}z_{Ap}\right)- (d+c z)z_{aP}\\
\dot{z}_{ap}&=& b z_{ap}+ \frac{b}{z}\left[ \beta_1 z_{ap}\frac{z_{aP}}{2}-\beta_2\left( z_{ap}  \frac{z_{AP}}{4}+ z_{aP}
 \frac{z_{Ap}}{4} \right) \right]\\
&& - \frac{b}{2z} \left(z_{AP}z_{ap}-z_{aP}z_{Ap}\right)- (d+c z)z_{ap}
\end{array}
\right. 
\end{equation} 
where $z=z_{AP}+z_{Ap}+z_{aP}+z_{ap}$ is the total size of the population. 

\begin{prop}
\label{prop:equilibre}
The dynamical system \eqref{eq:systbis} admits the following equilibria, with at least a null coordinate:
\begin{description}
\item[$(i)$] The state $(0,0,0,0)$, which is unstable.
\item[$(ii)$] Any state where only allele $p$ remains at locus $2$,
$$\left(0,\proportion\frac{b-d}{c},0,(1-\proportion)\frac{b-d}{c}\right),\quad \proportion\in[0,1].$$ 
The stability of these equilibria has been studied in Proposition \ref{prop:invasibilite}.
\item[$(iii)$] The three following states for which only allele $P$ remains at locus $2$
$$\mathbf{\chi}_{AP}=\left(\frac{(1+\beta_1) b-d}{c},0,0,0\right),\quad\quad \mathbf{\chi}_{aP}=\left(0,0,\frac{(1+\beta_1) b-d}{c},0\right) $$
and $$ \left(\frac{b(1+ (\beta_1-\beta_2)/2)-d}{2c},0,\frac{b(1+ (\beta_1-\beta_2)/2)-d}{2c},0\right).$$
The first two equilibria are stable, whereas the last one is unstable.
\end{description}
\end{prop}

\begin{proof} 
\begin{description}
\item[$(i)$] The state $(0,0,0,0)$ is an equilibrium, from Equation \eqref{eq:systbis}.
To prove that it is unstable, let us consider $\varepsilon>0$ and assume that the initial 
condition $\zbf_0$ satisfies $z(0)=\norm{\zbf_0}{1}\le\varepsilon$. 
We denote by $t_\eps=\inf\{t\ge0, z(t)>\varepsilon\}$ which would be infinite if $\zbf$ lies in the basin of attraction of $(0,0,0,0)$. From \eqref{eq:systbis}, we find
$$
\dot z_A-\dot z_a = \left(b-d-c z \right) (z_A-z_a)+\frac{b\beta_1}{z}(z_{AP}z_A-z_{aP}z_a).
$$
Thus, we obtain that $\forall t\le t_\eps$, 
$$\dot{z}_A-\dot{z}_a\ge (z_A-z_a)(b-d-c\varepsilon)-b\beta_1\varepsilon.$$
Let $\phi$ be the unique solution to the linear differential equation 
$$\dot{\phi}=\phi(b-d-c\varepsilon)-b\beta_1\varepsilon.$$ 
Then 
$$\phi(t)=(\phi(0)-\varepsilon(b-d-c\varepsilon)^{-1})e^{(b-d-c\varepsilon)t} +b\beta_1\varepsilon(b-d-c\varepsilon)^{-1}.$$
Using classical results on differential inequalities we deduce that if $(z_A-z_a)(0)=\phi(0)$ then for all $t\le t_\eps$, 
$z(t)\geq (z_A-z_a)(t)\ge\phi(t)$. 
Since for $\varepsilon$ small enough $\phi(t)\to\infty$ as $t\to\infty$, we deduce that $t_\eps$ is finite. 
In other words, $(0,0,0,0)$ is unstable.
\item[$(ii)$] Let us assume that $z_{AP}=z_{aP}=0$. Then, \eqref{eq:systbis} can be reduced to
\begin{equation*}
\left\{
\begin{aligned}
\dot z_{Ap}=(b-d-cz)z_{Ap}\\
\dot z_{ap}=(b-d-cz)z_{ap}.
\end{aligned}
\right.
\end{equation*}
Therefore, the set of points $(0,z_{Ap},0,(b-d)/c-z_{Ap})$ with $z_{Ap}\in[0,(b-d)/c]$ corresponds to the set of non null 
equilibria such that $z_{AP}=z_{aP}=0$.
\item[$(iii)$] Let us assume that $z_{Ap}=z_{ap}=0$. Then from \eqref{eq:systbis},
 \begin{align} \label{eq1patch1}
 \dot{z}_{AP}=0 &= z_{AP}\left( (b-d-cz)+ \frac{b\beta_1}{z}z_{AP}- \frac{b\beta_2}{z}z_{aP} \right),
\end{align}
and
\begin{align} \label{eq1patch2}
\dot{z}_{aP}= 0 &= z_{aP}\left( (b-d-cz)+ \frac{b\beta_1}{z}z_{aP}- \frac{b\beta_2}{z}z_{AP} \right).
\end{align}

If $z_{AP}=0$, then $z_{aP}= ((1+\beta_1) b-d)/c$, and similarly when exchanging $A$ and $a$.
For these equilibria, the eigenvalues of the Jacobian matrix are:
$$ \left(-\frac{b \beta_1}{2}, -b (\beta_1 + \beta_2), -\frac{b}{4} (2 + 4 \beta_1 + \beta_2), -b (1 + \beta_1) + d\right).  $$
Since $b>d$, these eigenvalues are negative and these equilibria are therefore stable. 

If $z_{AP}>0$ and $z_{aP}>0$ then by dividing \eqref{eq1patch1} by $z_{AP}$ and
\eqref{eq1patch2} by $z_{aP}$ and making the difference between both expressions, we get:
$$ \frac{b}{z}\left(\beta_1+\beta_2 \right)(z_{AP}-z_{aP})=0. $$
Then 
$$ z_{AP}= z_{aP}= \frac{b(1+ (\beta_1-\beta_2)/2)-d}{2c} \quad\quad\text{from \eqref{eq1patch1}.}$$
The eigenvalues of the Jacobian matrix in this equilibrium are:
$$ \left(\frac{b}{2} (\beta_1 + \beta_2), 
 \frac{b}{4} (\beta_2-\beta_1), -\frac{b}{4} (2 + \beta_1 - 2 \beta_2), -\frac{b}{2} (2 + \beta_1 - \beta_2) + 
  d\right). $$
The first eigenvalue is positive, therefore this equilibrium is unstable.
\end{description}

 We finally show that there is no other equilibrium with at least a null coordinate. 
To this aim, we first consider the case where $z_a=0$. Then
$$
\dot z_A =(b-d-cz)z_A +\frac{b\beta_1}{z}z_{AP}z_A=0,
$$
 and ${z}_A=0$ (which is the trivial equilibrium $(i)$) or $b-d-cz+b\beta_1z_{AP}/z=0.$ In the second case, from the equation satisfied by $z_{AP}$, we deduce that $b\beta_1z_{AP}z_{Ap}/(2z)=0 $. Hence
  either $z_{AP}=0$ or $z_{aP}=0$ (which corresponds to Equilibrium $(ii)$). Similar equilibria are retrieved by assuming $z_A=0$.\\

Finally consider the case where $z_{ap}=0$. Then from the equation satisfied by $z_{ap}$ given in \eqref{eq:systbis}, $$\dot{z}_{ap}=0=\frac{b}{2z}z_{aP}z_{Ap}(1-\frac{\beta_2}{2}).$$ Therefore $z_{aP}=0$ (then $z_{a}=0$ which corresponds to the case that has just been considered) or $z_{Ap}=0$ (which corresponds to Equilibrium $(iii)$). Similar arguments can be made assuming $z_{Ap}=0$ or $z_{aP}=0$ or $z_{AP}=0$.

\end{proof}

\subsection{Proof of Proposition \ref{prop:proba_ext}}
\label{proof:proba_ext}

In the particular case where $\proportion_A=1$, the transition rates of the bitype branching process $\bar{\textbf{N}}$ are equal to
$$ \bar{\beta}_{AA}=\frac{b}{2} (2+\beta_1), \quad \bar{\beta}_{aa}=\bar{\beta}_{aA}=\frac{b}{2}\left( 1- \frac{\beta_2}{2} \right)\quad \text{and} \quad 
\bar{\beta}_{Aa}=0, $$
and the system \eqref{syst_prob_ext} giving the extinction probabilities of the branching process $\bar{\textbf{N}}$ takes the simpler form:
\begin{align*}
u_A(s_A,s_a) &= b(1-s_A)+ \frac{b}{2}( 2+\beta_1 )(s_A^2-s_A) \\
u_a(s_A,s_a) &= b(1-s_a)+ \frac{b}{2} \left( 1 - \frac{\beta_2}{2} \right)(s_a^2-s_a)
+ \frac{b}{2}\left(1- \frac{\beta_2}{2}\right)(s_As_a-s_a).
\end{align*}
Recall that the extinction probabilities we are looking for are the smallest solution to 
$u_\alpha(s_A,s_a)=0, \ \alpha \in \mathfrak{A}$. 
We easily obtain from the first, linear, equation an expression of $q_A$. Then replacing it with its expression in the second equation gives that $q_a$ is the root of a second order polynomial function. This gives the result.

\subsection{Probabilistic technical results}
\label{app:proba}

\begin{lem} \label{lemnedescendpas}
 Let $(\bar N_A,\bar N_a)$ be a two type supercritical birth and death process. We recall that, for 
 $\alpha \in \mathfrak{A}$, 
 $q_{\alpha}$ is the extinction probability of the process when the initial individual is of type $\alpha$
 $$ q_{\alpha}= \P\Big(\exists t>0, \bar N_A(t)+\bar N_a(t)=0 \Big|(\bar N_A(0),\bar N_a(0))= \mathbf{e}_\alpha \Big)<1, $$
 where we denote by $(\mathbf{e}_A,\mathbf{e}_a)$ the canonical basis of $\R^2$. Let $C>2$ satisfying 
 \begin{equation} \label{assN}  
 C\left( \frac{\max\{q_{A}, q_{a} \} }{C-1} \right)^{1-1/C} <1.
 \end{equation}
Then 
 $$ 
 \lim_{k \to \infty}  \P\Big( S_{\lfloor k/C \rfloor }<\infty |\bar N_A(0)+\bar N_a(0)=k\Big)=0, 
 $$
where the stopping time $S_l$ is defined for any $l\in \N$ by
$$ 
S_l := \inf \{ t \geq 0, \bar N_A(t)+\bar N_a(t)=l  \}. 
$$
\end{lem}

\begin{proof}
Let us first remark that it is possible to choose such a constant $C>2$ since the map $x\mapsto x(\max(q_A,q_a)/(x-1))^{1-1/x}$ is continuous and goes to $\max(q_A,q_a)<1$ as $x\to\infty$.\\
There are initially $k$ individuals and we want to lower bound the probability that the population size reaches $\lfloor k/C\rfloor$.
If this happens at a finite time, then it means that, at least, $k-\lfloor k/C\rfloor$ individuals alive at time $0$ have a finite line of descent. 
But we know that each individual has a finite line of descent with a probability smaller than $q:=\max(q_{A}, q_{a} )$. Then using the branching property, the probability that exactly $i$ initial individual out of $k$ have a finite line of descent is smaller than $\binom{k}{i} q^{i}(1-q)^{k-i}$. Hence 
$$ \P\Big( S_{\lfloor k/C \rfloor }<\infty |\bar N_A(0)+\bar N_a(0)=k\Big)\leq \sum_{i=k-\lfloor k/C \rfloor}^k \binom{k}{i} q^{i}(1-q)^{k-i}.$$
Since $i\mapsto q^i$ and $i\mapsto \binom{k}{i}$ are decreasing functions as soon as $i\geq k/2$ , we deduce that for $C>2$ and $k$ large
\begin{align*}
        \P\Big( S_{\lfloor k/C \rfloor }<\infty |\bar N_A(0)+\bar N_a(0)=k\Big)      & \leq \binom{k}{k-\lfloor k/C \rfloor} q^{k-\lfloor k/C \rfloor} \sum_{i=k-\lfloor k/C \rfloor}^k (1-q)^{k-i}\\
       & \leq \frac{1}{q} \frac{k!}{(\lfloor k/C \rfloor)!(k-\lfloor k/C \rfloor)!}
       q^{k-\lfloor k/C \rfloor}.
       \end{align*}
Moreover, using Stirling's formula we get
       \begin{align*}
       &\frac{k!}{(\lfloor k/C \rfloor)!(k-\lfloor k/C \rfloor)!}
       q^{k-\lfloor k/C \rfloor}\\
        &\underset{k\to \infty}{\sim} \frac{\sqrt{k}k^k q^{k-\lfloor k/C \rfloor}}{\sqrt{2\pi \lfloor k/C \rfloor(k-\lfloor k/C \rfloor)}\lfloor k/C \rfloor^{\lfloor k/C \rfloor}
        (k-\lfloor k/C \rfloor)^{k-\lfloor k/C \rfloor}}\\
        &\underset{k\to \infty}{\sim}  \sqrt{\frac{C}{2\pi k (1-1/C)}}\left( \frac{k q}{k-\lfloor k/C \rfloor} \right)^k 
        \left( \frac{k-\lfloor k/C \rfloor}{\lfloor k/C \rfloor(1-\lambda_A \wedge \lambda_a)} \right)^{\lfloor k/C \rfloor} \underset{k\to \infty}{\to} 0, 
      \end{align*}
under assumption \eqref{assN}.
This ends the proof.
      \end{proof}

\begin{lem}\label{yule}
Let us consider a one dimensional pure birth process $X$ with birth rate $b$. Denote for $k>0$ by $\tau_k$
the hitting time of $\lfloor k \rfloor$ by the process $X$. Then there exists a finite $C$ such that 
$$ \limsup_{K \to \infty} \P\Big(\tau_{\sqrt{\eps}K}<  \tau_{{\eps}K}+ \ln \ln 1/\eps\Big) \leq C \sqrt{\eps} (\ln 1/\eps)^{b}.  $$
\end{lem}

\begin{proof}
Using the Markov property of the process, we find
 \begin{align*}
 \P\Big(\tau_{\sqrt{\eps}K}<  \tau_{{\eps}K} +\ln \ln 1/\eps\Big)&=  \P\Big(\tau_{\sqrt{\eps}K}< \ln \ln 1/\eps \Big| X(0)=\lfloor \eps K \rfloor\Big)\\
 & =\P\Big(X(\tau_{\sqrt{\eps}K}) e^{-b\tau_{\sqrt{\eps}K}}
  > \lfloor\sqrt{\eps}K\rfloor(\ln 1/\eps)^{-b}\Big| X(0)=\lfloor \eps K \rfloor\Big).
  \end{align*}
Now using Markov Inequality and the fact that conditioning on $\{X(0)=\lfloor\varepsilon K\rfloor\}$, $X(t) e^{-bt}$ is a martingale with expectation $\lfloor \eps K \rfloor$, we obtain 
 \begin{equation*}
  \P\Big(\tau_{\sqrt{\eps}K}<  \tau_{{\eps}K}+ \ln \ln 1/\eps\Big) \leq \frac{\lfloor \eps K \rfloor}{\lfloor\sqrt{\eps}K\rfloor(\ln 1/\eps)^{-b}}
  \sim \sqrt{\eps}(\ln 1/\eps)^{b}, \quad \text{for} \quad K \to \infty.
 \end{equation*}
 This concludes the proof.
\end{proof}

\begin{lem}\label{lem:reg_proba_exinc}
Let us consider a family of two-type branching processes $(\bar{N}_A^\varepsilon,\bar{N}_a^\varepsilon, \eps\in\mathbb{R})$ whose transition rates are given by
\begin{align*}
 (\bar{N}_A^\varepsilon,\bar{N}_a^\varepsilon) \to (\bar{N}_A^\varepsilon+1,\bar{N}_a^\varepsilon)  \quad &\text{at rate} \quad b_{AA}^\varepsilon\bar{N}_A^\varepsilon+b^\varepsilon_{aA}\bar{N}_a^\varepsilon, \\ 
 (\bar{N}_A^\varepsilon,\bar{N}_a^\varepsilon) \to (\bar{N}_A^\varepsilon,\bar{N}_a^\varepsilon+1) \quad &\text{at rate} \quad b^\varepsilon_{Aa}\bar{N}_A^\varepsilon+b^\varepsilon_{aa}\bar{N}^\varepsilon_a, \\ 
 (\bar{N}_A^\varepsilon,\bar{N}_a^\varepsilon) \to (\bar{N}_A^\varepsilon-1,\bar{N}_a^\varepsilon) \quad &\text{at rate} \quad d^\varepsilon\bar{N}_A^\varepsilon, \\ 
 (\bar{N}_A^\varepsilon,\bar{N}_a^\varepsilon) \to (\bar{N}_A^\varepsilon,\bar{N}_a^\varepsilon-1) \quad &\text{at rate} \quad d^\varepsilon\bar{N}_a^\varepsilon,
\end{align*}
and let us denote by $q^\eps=(q_A^\eps,q_a^\eps)$ the extinction probabilities of the process $\bar{N}^\varepsilon$
with initial state an individual of type $A$ or $a$.\\
$(i)$ Let us assume that the functions 
$\varepsilon\mapsto b_i^\varepsilon>0$ for 
$i\in \mathfrak{A}^2$ (resp $\varepsilon\mapsto d^\varepsilon>0$) are of class $\mathcal{C}^k$ for $k\geq0$ in $\eps=0$ and 
that the process $(\bar{N}^0_A,\bar{N}^0_a)$ is supercritical. 
Then the application $\varepsilon\mapsto  \mathbf{q}^\varepsilon$ is of class $\mathcal{C}^k$ 
in $\varepsilon=0$. \\
$(ii)$ Let us assume furthermore that the functions $\varepsilon\mapsto b_i^\varepsilon>0$ for $
i\in \mathfrak{A}^2$ (resp $\varepsilon\mapsto d^\varepsilon>0$) are non decreasing (resp. non increasing), and consider
$\varepsilon_1\le\varepsilon_2$ then the extinction probabilities $\mathbf{q}^i=(q_A^{\varepsilon_i},q_a^{\varepsilon_i})$ 
($i\in\{1,2\}$) of the two branching processes $(\bar{N}^{\varepsilon_i}_A,\bar{N}^{\varepsilon_i}_a)$ satisfy 
$$\mathbf{q}^1\le \mathbf{q}^2,$$
where the inequality applies to both coordinates.
\end{lem}

\begin{proof}
$(i)$ The proof relies on Theorem 6.2 of \cite{RughAlili2008} that considers multi-type discrete time branching processes. 
The process $(\bar{N}_A^\varepsilon,\bar{N}_a^\varepsilon, \varepsilon\in\mathbb{R})$ is a continuous time multi-type linear 
birth-and-death process in which for all $\alpha_1,\alpha_2\in\mathcal{A}_0$, individuals with genotype $\alpha_1$ die at rate 
$d^{\varepsilon}$ and produce an offspring with genotype $\alpha_2$ at rate $b_{\alpha_1,\alpha_2}^{\varepsilon}$. For all 
$\alpha_1,\alpha_2\in\mathcal{A}_0$, the random variable $N_{\alpha_1,\alpha_2}$ giving the number of offsprings of type $\alpha_2$ 
of a given individual of type $\alpha_1$ satisfies 
$$\mathbb{P}(N_{\alpha_1,\alpha_2}=k)=\left(\frac{b^{\varepsilon}_{\alpha_1\alpha_2}}{b^{\varepsilon}_{\alpha_1A}+
b^{\varepsilon}_{\alpha_1a}+d^{\varepsilon}}\right)^k\frac{d^{\varepsilon}}{b^{\varepsilon}_{\alpha_1A}+
b^{\varepsilon}_{\alpha_1a}+d^{\varepsilon}}$$ which is assumed to be $\mathcal{C}^k$ in $\epsilon$ at $0$.
Let us consider 
the discrete time stochastic process with values in $\N^2$ giving the number of individuals of each type at each generation,
whose extinction probability is equal to $\mathbf{q}^{\varepsilon}$. Then for any $(s_A,s_a)\in [0,1)^2$, 
$$\mathbb{E}(s_A^{N_{\alpha,A}}s_a^{N_{\alpha,a}})=\frac{d^{\epsilon}}{b^{\epsilon}_{\alpha A}+b^{\epsilon}_{\alpha a}+d^{\epsilon}-b^{\epsilon}_{\alpha A}s_1-b^{\epsilon}_{\alpha a}s_2}.$$
To apply Theorem 6.2 of \cite{RughAlili2008}, we therefore need to find $s=(s_A,s_a)\in [0,1)^2$ such that 
$$\mathbb{E}(s_A^{N_{\alpha,A}}s_{a}^{N_{\alpha,a}})<s_{\alpha} \quad  \text{for all $\alpha\in \mathfrak{A}$}.$$ 
This condition is sufficient to check Assumption 6.1 of \cite{RughAlili2008} (in which $\mathbf{\alpha}$ is now denoted $(s_1,s_2)$) because we use here the particular framework 
of constant environment. Therefore, by taking 
$$\tilde{s}_{\alpha}=\frac{\mathbb{E}(s_A^{N_{\alpha,A}}s_{a}^{N_{\alpha,a}})+s_{\alpha}}{2},$$ we get 
$$\mathbb{E}(s_{A}^{N_{\alpha,A}}s_{a}^{N_{\alpha,a}})<\tilde{s}_{\alpha}<s_{\alpha} \quad  
\text{for all $\alpha\in \mathfrak{A}^2$}$$
which is exactly Assumption 6.1 of \cite{RughAlili2008}.
For any $(s_1,s_2)\in[0,1)^2$, let 
$$\phi(s_1,s_2)=(s_1(b_{11}+b_{12}+d-s_1b_{11}-s_2b_{12})-d,s_2(b_{21}+b_{22}+d-s_1b_{21}-s_2b_{22})-d).$$ 
We have $\phi(1,1)=(0,0)$ and we seek $(s_1,s_2)\in [0,1)^2$ such that $\phi(s_1,s_2)>(0,0)$ where the inequality applies 
to both coordinates. This is possible if the jabobian matrix of the application $\phi$ in $(1,1)$ which is equal to 
$$\begin{pmatrix}
d-b_{11} & -b_{12} \\ 
-b_{21} & d-b_{22} 
\end{pmatrix} $$ has a negative eigenvalue and this condition is equivalent to the supercriticality of the process 
$(\bar{N}^0_A,\bar{N}^0_a)$.\\
$(ii)$ The proof relies on a coupling argument. Let us construct the two processes using the same Poisson point measures. 
Then we have that almost surely 
$$\bar{N}^1_A\ge \bar{N}^2_A \quad \text{ and}\quad \bar{N}^1_a\ge\bar{N}^2_a.$$
Therefore, for every $t\ge0$,
$$\P(\bar{N}^1(t)=0)\le \P(\bar{N}^2(t)=0),$$
which gives the result, by letting $t\to\infty$.

\end{proof}

\section{Table for birth rates}\label{appTable}
In this table we present the birth rates and the possible offspring of every couples in the population. When a $P$ individual is involved, we differentiate whether it is the choosing parent (1st parent) or the chosen one (2nd parent).
Let us briefly recall how the table is constructed.\\
For the possible offspring, we assume Mendelian reproduction meaning that for each gene independently  an allele is chosen at random among the two alleles of the parent. As an example, in a mating $AP\times Ap$ the offspring will necessary receive allele $A$ and then choose with equal probability between $p$ and $P$, and we note in the third column $1/2 AP$ and $1/2 Ap$.
For the birth rate of the same couple, since the choosing parent carries allele $P$, mating occurs with a preference at rate $b(1+\beta_1)$ since both parent carry allele $A$. 
\smallskip

\small
\begin{center}
   \tablefirsthead{\hline  \multicolumn{1}{|c}{1st parent}
   & \multicolumn{1}{c|}{2nd parent}
   & Descendant
   & \multicolumn{1}{c|}{Rate} \\ \hline}
   \tablehead{ \multicolumn{1}{|c}{1st parent}
   & \multicolumn{1}{c|}{2nd parent}
   & Descendant
   & \multicolumn{1}{c|}{Rate} \\ \hline}
   \tabletail{}
   \tablelasttail{\hline}
   \bottomcaption{This table gives the rates at which two parents with given genotypes give birth to an offspring with given genotype, for all possible values of these genotypes. By convention, the first parent is assumed to be responsible for homogamy, when carrying allele $P$.\label{tab:repro}}
\par

\begin{supertabular}{|cc|c|c|}
Ap&Ap&Ap&$b\frac{n_{Ap}n_{Ap}}{n}$\\
\hline
ap&ap&ap&$b\frac{n_{ap}n_{ap}}{n}$\\
\hline
ap&Ap&$\frac{1}{2}$ap&$b\frac{n_{ap}n_{Ap}}{n}$ \\
& &$\frac{1}{2}$Ap&$b\frac{n_{ap}n_{Ap}}{n}$\\
\hline
Ap&ap&$\frac{1}{2}$ap&$b\frac{n_{ap}n_{Ap}}{n}$ \\
& &$\frac{1}{2}$Ap&$b\frac{n_{ap}n_{Ap}}{n}$\\
\hline
AP&AP&AP&$b(1+\beta_1)\frac{n_{AP}n_{AP}}{n}$\\
\hline
aP&aP&aP&$b(1+\beta_1)\frac{n_{aP}n_{aP}}{n}$\\
\hline
aP&AP&$\frac{1}{2}$aP&$b(1-\beta_2)\frac{n_{aP}n_{AP}}{n}$ \\
& &$\frac{1}{2}$AP&$b(1-\beta_2)\frac{n_{aP}n_{AP}}{n}$\\
\hline
AP&aP&$\frac{1}{2}$aP&$b(1-\beta_2)\frac{n_{aP}n_{AP}}{n}$ \\
& &$\frac{1}{2}$AP&$b(1-\beta_2)\frac{n_{aP}n_{AP}}{n}$\\
\hline
AP&Ap&$\frac{1}{2}$AP& $b(1+\beta_1)\frac{n_{AP}n_{Ap}}{n}$\\
&&$\frac{1}{2}$Ap& $b(1+\beta_1)\frac{n_{AP}n_{Ap}}{n}$\\
\hline
Ap&AP&$\frac{1}{2}$AP& $b\frac{n_{AP}n_{Ap}}{n}$\\
&&$\frac{1}{2}$Ap& $b\frac{n_{AP}n_{Ap}}{n}$\\
\hline
aP&ap&$\frac{1}{2}$aP& $b(1+\beta_1)\frac{n_{aP}n_{ap}}{n}$\\
&&$\frac{1}{2}$ap& $b(1+\beta_1)\frac{n_{aP}n_{ap}}{n}$\\
\hline
ap&aP&$\frac{1}{2}$aP& $b\frac{n_{aP}n_{ap}}{n}$\\
&&$\frac{1}{2}$ap& $b\frac{n_{aP}n_{ap}}{n}$\\
\hline
AP&ap&$\frac{1}{4}$AP& $b(1-\beta_2)\frac{n_{AP}n_{ap}}{n}$\\
&&$\frac{1}{4}$Ap& $b(1-\beta_2)\frac{n_{AP}n_{ap}}{n}$\\
&&$\frac{1}{4}$aP& $b(1-\beta_2)\frac{n_{AP}n_{ap}}{n}$\\
&&$\frac{1}{4}$ap& $b(1-\beta_2)\frac{n_{AP}n_{ap}}{n}$\\
\hline
ap&AP&$\frac{1}{4}$AP& $b\frac{n_{AP}n_{ap}}{n}$\\
&&$\frac{1}{4}$Ap& $b\frac{n_{AP}n_{ap}}{n}$\\
&&$\frac{1}{4}$aP& $b\frac{n_{AP}n_{ap}}{n}$\\
&&$\frac{1}{4}$ap& $b\frac{n_{AP}n_{ap}}{n}$\\
\hline
aP&Ap&$\frac{1}{4}$AP& $b(1-\beta_2)\frac{n_{aP}n_{Ap}}{n}$\\
&&$\frac{1}{4}$Ap& $b(1-\beta_2)\frac{n_{aP}n_{Ap}}{n}$\\
&&$\frac{1}{4}$aP& $b(1-\beta_2)\frac{n_{aP}n_{Ap}}{n}$\\
&&$\frac{1}{4}$ap& $b(1-\beta_2)\frac{n_{aP}n_{Ap}}{n}$\\
\hline
Ap&aP&$\frac{1}{4}$AP& $b\frac{n_{aP}n_{Ap}}{n}$\\
&&$\frac{1}{4}$Ap& $b\frac{n_{aP}n_{Ap}}{n}$\\
&&$\frac{1}{4}$aP& $b\frac{n_{aP}n_{Ap}}{n}$\\
&&$\frac{1}{4}$ap& $b\frac{n_{aP}n_{Ap}}{n}$\\
\end{supertabular}
\end{center}

\normalsize
\section*{Acknowledgments}
The authors thank the CNRS for its financial support through its competitive funding programs on interdisciplinary research.
This work was partially funded by the Chair "Mod\'elisation Math\'ematique et Biodiversit\'e" of VEOLIA-Ecole Polytechnique-MNHN-F.X.
H.L. acknowledges support from CONACyT-MEXICO and the foundation Sof\'ia Kovalevskaia of SMM.
The authors are grateful to Andr\'as T\'obi\'as for his careful reading of the paper and his useful comments.

 \bibliographystyle{abbrv}
 \bibliography{biblio_mating_pref}

\begin{thebibliography}{10}

\bibitem{RughAlili2008}
S.~Alili and H.~H. Rugh.
\newblock On the regularity of the extinction probability of a branching
  process in varying and random environments.
\newblock {\em Nonlinearity}, 21(2):353, 2008.

\bibitem{athreya1972branching}
K.~B. Athreya and P.~E. Ney.
\newblock {\em Branching processes}.
\newblock Springer-Verlag Berlin, Mineola, NY, 1972.
\newblock Reprint of the 1972 original [Springer, New York; MR0373040].

\bibitem{billiard2017interplay}
S.~Billiard and C.~Smadi.
\newblock The interplay of two mutations in a population of varying size: a
  stochastic eco-evolutionary model for clonal interference.
\newblock {\em Stochastic Processes and their Applications}, 127(3):701--748,
  2017.

\bibitem{bolker1997using}
B.~Bolker and S.~W. Pacala.
\newblock Using moment equations to understand stochastically driven spatial
  pattern formation in ecological systems.
\newblock {\em Theoretical population biology}, 52(3):179--197, 1997.

\bibitem{champagnat2006microscopic}
N.~Champagnat.
\newblock A microscopic interpretation for adaptive dynamics trait substitution
  sequence models.
\newblock {\em Stochastic processes and their applications}, 116(8):1127--1160,
  2006.

\bibitem{champagnat2011polymorphic}
N.~Champagnat and S.~M{\'e}l{\'e}ard.
\newblock Polymorphic evolution sequence and evolutionary branching.
\newblock {\em Probability Theory and Related Fields}, 151(1-2):45--94, 2011.

\bibitem{chicone2006ode}
C.~Chicone.
\newblock {\em Ordinary Differential Equations with Applications}.
\newblock Number~34 in Texts in Applied Mathematics. Springer-Verlag New York,
  2006.

\bibitem{collet2011rigorous}
P.~Collet, S.~M{\'e}l{\'e}ard, and J.~A. Metz.
\newblock A rigorous model study of the adaptive dynamics of mendelian
  diploids.
\newblock {\em Journal of Mathematical Biology}, 67(3):569--607, 2013.

\bibitem{coron2015slow}
C.~Coron.
\newblock Slow-fast stochastic diffusion dynamics and quasi-stationary
  distributions for diploid populations.
\newblock {\em J. Math. Biol.}, Published Online, 2013.

\bibitem{Coron2018}
C.~Coron, M.~Costa, H.~Leman, and C.~Smadi.
\newblock A stochastic model for speciation by mating preferences.
\newblock {\em Journal of Mathematical Biology}, 76(6):1421--1463, May 2018.

\bibitem{dieckmann1996dynamical}
U.~Dieckmann and R.~Law.
\newblock The dynamical theory of coevolution: a derivation from stochastic
  ecological processes.
\newblock {\em Journal of mathematical biology}, 34(5-6):579--612, 1996.

\bibitem{dupuis2011weak}
P.~Dupuis and R.~S. Ellis.
\newblock {\em A weak convergence approach to the theory of large deviations},
  volume 902.
\newblock John Wiley \& Sons, 2011.

\bibitem{Ethier-Kurtz}
S.~N. Ethier and T.~G. Kurtz.
\newblock {\em Markov processes. Characterization and convergence}.
\newblock Wiley Series in Probability and Mathematical Statistics: Probability
  and Mathematical Statistics. John Wiley \& Sons Inc., New York, 1986.

\bibitem{fournier2004microscopic}
N.~Fournier and S.~M{\'e}l{\'e}ard.
\newblock A microscopic probabilistic description of a locally regulated
  population and macroscopic approximations.
\newblock {\em Annals of applied probability}, pages 1880--1919, 2004.

\bibitem{freidlin1984random}
M.~Freidlin and A.~D. Wentzell.
\newblock {\em Random perturbations of dynamical systems}, volume 260.
\newblock Springer, 1984.

\bibitem{georgii2003supercritical}
H.-O. Georgii and E.~Baake.
\newblock Supercritical multitype branching processes: the ancestral types of
  typical individuals.
\newblock {\em Advances in Applied Probability}, 35(4):1090--1110, 2003.

\bibitem{gregorius1992two}
H.-R. Gregorius.
\newblock A two-locus model of speciation.
\newblock {\em Journal of theoretical Biology}, 154(3):391--398, 1992.

\bibitem{heinzmann2009extinction}
D.~Heinzmann et~al.
\newblock Extinction times in multitype markov branching processes.
\newblock {\em Journal of Applied Probability}, 46(1):296--307, 2009.

\bibitem{Herrero2003}
M.~Herrero.
\newblock Male and female synchrony and the regulation of mating in flowering
  plants.
\newblock {\em Philosophical Transactions of the Royal Society B: Biological
  Sciences}, 358:1019--1024, 2003.

\bibitem{jiang2013assortative}
Y.~Jiang, D.~Bolnick, and M.~Kirkpatrick.
\newblock Assortative mating in animals.
\newblock {\em The American Naturalist}, 181(6):E125--E138, 2013.

\bibitem{leman2018stochastic}
H.~Leman.
\newblock A stochastic model for reproductive isolation under asymmetrical
  mating preferences.
\newblock {\em Bulletin of mathematical biology}, 80(9):2502--2525, 2018.

\bibitem{mclain1987male}
D.~McLain and R.~Boromisa.
\newblock Male choice, fighting ability, assortative mating and the intensity
  of sexual selection in the milkweed longhorn beetle, tetraopes
  tetraophthalmus (coleoptera, cerambycidae).
\newblock {\em Behavioral Ecology and Sociobiology}, 20(4):239--246, 1987.

\bibitem{neukirch2017survival}
R.~Neukirch and A.~Bovier.
\newblock Survival of a recessive allele in a mendelian diploid model.
\newblock {\em Journal of mathematical biology}, 75(1):145--198, 2017.

\bibitem{Savolainenetal2006}
V.~Savolainen, M.~Anstett, C.~Lexer, I.~Hutton, J.~Clarkson, M.~Norup,
  M.~Powell, D.~Springate, N.~Salamin, and W.~Baker.
\newblock Sympatric speciation in palms on an oceanic island.
\newblock {\em Nature}, 441:210--213, 2006.

\bibitem{smadi2015eco}
C.~Smadi.
\newblock An eco-evolutionary approach of adaptation and recombination in a
  large population of varying size.
\newblock {\em Stochastic Processes and their Applications}, 2015.

\bibitem{smadi2018looking}
C.~Smadi, H.~Leman, and V.~Llaurens.
\newblock Looking for the right mate in diploid species: How does genetic
  dominance affect the spatial differentiation of a sexual trait?
\newblock {\em Journal of theoretical biology}, 447:154--170, 2018.

\end{thebibliography}

\end{document}